\documentclass[11pt]{article}

\usepackage{amsmath,amsfonts,amsthm,amssymb,verbatim,graphicx,subfigure,color,fullpage,multirow}
\usepackage{pgf,tikz,pgfplots} 

\newtheorem{lemma}{Lemma}
\newtheorem{theorem}{Theorem}
\newtheorem*{theorem*}{Theorem}
\newtheorem{proposition}{Proposition}
\newtheorem{corollary}{Corollary}

\newtheorem{conjecture}{Conjecture}

\newtheorem{observation}{Observation}
\newtheorem{example}{Example}

\title{On the diameter of Schrijver graphs\footnote{Partially supported by LIA
INFINIS/SINFIN (CNRS-CONICET-UBA, France--Argentine) and by the MathAmSud projet MATHAMSUD 20-MATH-09 (France - Argentina - Chile).} \footnote{A short version of this paper appears in Proc. of Lagos'21 symposium, Procedia Comp. Sci., 195:266-274,2021.}}

\author{Agustina Victoria Ledezma \footnote{Instituto de Matem\'atica Aplicada San Luis, Universidad Nacional de San Luis and CONICET, San Luis, Argentina. e-mail: agustinaledezma@gmail.com}
\and
Adri\'an Pastine \footnote{Instituto de Matem\'atica Aplicada San Luis, Universidad Nacional de San Luis and CONICET, San Luis, Argentina. e-mail: agpastine@gmail.com}
\and Pablo Torres \footnote{Depto. de Matem\'atica, Universidad Nacional de Rosario and CONICET, Rosario, Argentina. e-mail: ptorres@fceia.unr.edu.ar}
  \and Mario Valencia-Pabon\footnote{Universit\'e Sorbonne Paris Nord, LIPN, UMR7030, Villetaneuse, France. e-mail: valencia@lipn.univ-paris13.fr}}
\date{}

\begin{document}
\maketitle
\sloppy

\begin{abstract}
For $k \geq 1$ and $n \geq 2k$, the well known Kneser graph $\operatorname{KG}(n,k)$ has all $k$-element subsets of an $n$-element set as vertices; two such subsets are adjacent if they are disjoint.
Schrijver constructed a vertex-critical subgraph $\operatorname{SG}(n,k)$ of $\operatorname{KG}(n,k)$ with the same chromatic number. In this paper, we compute the diameter of the graph $\operatorname{SG}(2k+r,k)$ with $r \geq 1$. We obtain an exact value of the diameter of $\operatorname{SG}(2k+r,k)$ when $r \in \{1,2\}$ or when $r \geq k-3$. For the remained cases, when $3 \leq r \leq k-4$, we obtain that the diameter of $\operatorname{SG}(2k+r,k)$ belongs to the integer interval $[4..k-r-1]$.\\

\end{abstract}

\section{Introduction}\label{sec:intro}
Let $G$ be a connected graph. Given two vertices $a,b \in G$, $\text{dist}(a,b)$, the {\it distance} between $a$ and $b$, is defined as the length of the shortest path in $G$ joining $a$ to $b$. The {\it diameter} of $G$, that we denote by $D(G)$, is defined as the maximum distance between any pair of vertices in $G$.

Let $[n]$ denote the set $\{1,\ldots,n\}$. For positive integers $n \geq 2k$, the Kneser graph $\operatorname{KG}(n,k)$ has as vertices the $k$-subsets of $[n]$ and two vertices are connected by an edge if they have empty intersection.  In a famous paper, Lov\'asz \cite{Lov78} showed that its chromatic number $\chi(\text{KG}(n,k))$ is equal to $n-2k+2$. After this result, Schrijver \cite{Sch78} proved that the chromatic number remains the same when we consider the subgraph $\operatorname{KG}(n,k)_{2-\operatorname{stab}}$ of $\operatorname{KG}(n,k)$ obtained by restricting the vertex set to the $k$-subsets that are {\em $2$-stable}, that is, that do not contain two consecutive elements of $[n]$ (where $1$ and $n$ are considered also to be consecutive). Schrijver \cite{Sch78} also proved that the $2$-stable Kneser graphs are {\em vertex critical} (or {\em $\chi$-critical}), i.e. the chromatic number of any proper subgraph of $\operatorname{KG}(n,k)_{2-\operatorname{stab}}$ is strictly less than $n-2k+2$; for this reason, the $2$-stable Kneser graphs are also 
 known as Schrijver graphs. From now on we will use throughout this paper the notation $\operatorname{SG}(n,k)$ to refer to the graph $\operatorname{KG}(n,k)_{2-\operatorname{stab}}$.  

After these general advances, a lot of work has been done concerning properties of Kneser graphs and stable Kneser graphs (see \cite{Braun10,Braun11,GodRoy01,KaSte20,KaSte22,Meun11,SiTa20,Torres15,T-V17,V-V05} and references therein). Concerning Kneser graphs, its diameter was computed in \cite{V-V05}. Moreover, it is known that the distance between two vertices in Kneser graphs $\operatorname{KG}(n,k)$ only depends on the cardinality of their intersection \cite{V-V05}. However, in the case of Schrijver graphs $\operatorname{SG}(n,k)$ this does not work in the same way. For example, note that in $\operatorname{SG}(10,4)$ the vertices $\{1,3,5,7\}$ and $\{1,3,6,8\}$ are at distance $3$, while $\{1,3,6,8\}$ and $\{1,4,6,9\}$ are at distance $2$. In this paper, we are interested in computing the diameter of Schrijver graphs. As far as we know this parameter has not been studied for such graphs. The main result of this paper is the following theorem:

\begin{theorem}
\label{teo-ppal}
Let $n,k,r$ be positive integers such that $n = 2k+r$. Then, the diameter of the Schrijver graph $\text{SG}(2k+r,k)$ verifies
$$D(\text{SG}(2k+r,k)) \left\{
        \begin{array}{lcl}
             = 2  &;&  \text{if }r \geq 2k-2\text{, with }k \geq 2 \text{ (Theorem \ref{cor:s2diam2})}\\[5pt]
             = 3  &;&  \text{if }k-2 \leq r \leq 2k-3\text{, with }k \geq 3 \text{ (Theorem \ref{th:diam3})}\\[5pt]
             = 4  &;& \text{if }r = k-3\text{, with }k \geq 5 \text{ (Corollary \ref{cor:3k-3})}\\[5pt]         
              \in [4..k-r+1] &;&  \text{if }3 \leq r \leq k-4\text{, with }k \geq 5 \text{ (Corollary \ref{cor:3k-3} and Theorem \ref{th:boundm+3})}\\[5pt]
             = \left\lfloor\frac{3k}{4}\right\rfloor + (k\hspace*{-0.3cm}\mod 2) &;& \text{if }r = 2\text{, with }k \geq 6 \text{ (Theorem \ref{th:2k+2})}\\[5pt]
             = k &;& \text{if }r = 1 \text{ (Observation \ref{rem:diamd2})}
         \end{array}
       \right.$$
 
\end{theorem}
The proof of Theorem \ref{teo-ppal} will follow from Observation \ref{rem:diamd2}, Corollary \ref{cor:3k-3} and Theorems \ref{cor:s2diam2}, \ref{th:diam3}, \ref{th:boundm+3} and \ref {th:2k+2} given in the next sections.

\section{Main results}
A subset $S \subseteq [n]$ is {\em $s$-stable} if any two of its elements are at least "at distance $s$ apart'' on the $n$-cycle, that is, if $s \leq |i-j| \leq n-s$ for distinct $i,j \in S$. For $s,k \geq 2$ and $n \geq ks$, the $s$-stable Kneser graph $\operatorname{KG}(n,k)_{s-\operatorname{stab}}$ is the subgraph of $\operatorname{KG}(n,k)$ obtained by restricting the vertex set of $\operatorname{KG}(n,k)$ to the $s$-stable $k$-subsets of $[n]$. 

In \cite{T-V17} it was shown (see Proposition 4.3 in \cite{T-V17}) that $\operatorname{KG}(ks+1,k)_{s-\operatorname{stab}}$ is isomorphic to the complement graph of the $(k-1)$th power of a cycle $C_{ks+1}$. Therefore, $\text{SG}(2k+1,k)$ is isomorphic to a cycle graph $C_{2k+1}$ and so, we have the following straightforward observation.

\begin{observation}\label{rem:diamd2}
$\operatorname{D}(\operatorname{SG}(2k+1,k)) =k$.
\end{observation}

From now on, we assume that $n \geq 2k + 2$. We denote by $[n]_2$ the family of $2$-stable subsets of $[n]$ and by $[n]_2^k$ the family of $2$-stable $k$-subsets of $[n]$, i.e. $[n]_2^k=V\left(\operatorname{SG}(n,k)\right)$. 
We will always assume w.l.o.g. that any vertex $v=\{v_1,v_2,\ldots,v_k\}$ in $\operatorname{SG}(n,k)$ is such that $v_1 < v_2 < \ldots < v_k$. Arithmetic operations will be supposed modulo $n$ (being $0 \equiv n$).

\subsection{Distances between vertices}
Let $A=\{a_1,a_2,\dots,a_k\}$ and $B=\{b_1,b_2,\dots,b_k\}$ be two vertices in $\operatorname{SG}(n,k)$ such that $|A\cap B|=1$. W.l.o.g. we assume that $A\cap B=\{1\}$ and $a_2<b_2$. Note that $b_2\geq 4$. Let $X=\{2,b_2,b_3,\dots,b_k\}$ and $Y=X+1 = \{3,b_2+1,b_3+1,\dots,b_k+1\}$. It is not hard to see that the set of vertices $\{A,X,Y,B\}$ induce a $P_4$ or a paw (Figure \ref{P4andPaw}).

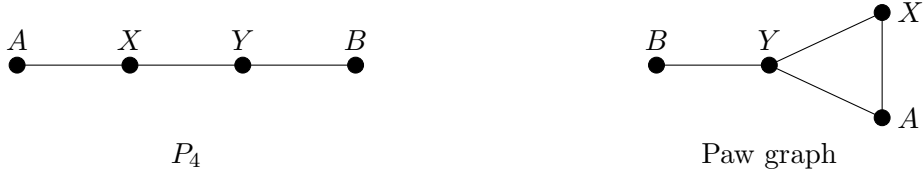
\begin{figure}[ht]
    \centering
    \begin{tikzpicture}
       
        \begin{scope}[shift={(-4,0)}]
        \draw (-3,0) -- (1.5,0);
        \filldraw (-3,0) circle (3pt) node[above, yshift=2] {$A$}; 
				\filldraw (-1.5,0) circle (3pt) node[above, yshift=2] {$X$};
				\filldraw (0,0) circle (3pt) node[above, yshift=2] {$Y$};
				\filldraw (1.5,0) circle (3pt) node[above, yshift=2] {$B$};
        \end{scope}
        \node at (-4.75,-1.2) {$P_4$};
        
        \begin{scope}[shift={(3,0)}]
        \draw (-1.5,0) -- (0,0);
        \draw (0,0) -- (1.5,0.7);
				\draw (0,0) -- (1.5,-0.7);
				\draw (1.5,0.7) -- (1.5,-0.7);
				\filldraw (-1.5,0) circle (3pt) node[above, yshift=2] {$B$};
				\filldraw (0,0) circle (3pt) node[above, yshift=2] {$Y$};
				\filldraw (1.5,0.7) circle (3pt) node[right, xshift=2] {$X$};
        \filldraw (1.5,-0.7) circle (3pt) node[right, xshift=2] {$A$};
        \end{scope}
        \node at (3,-1.2) {Paw graph};
    \end{tikzpicture}
    
    \caption{Graphs $P_4$ and paw}
    \label{P4andPaw}
\end{figure}

Assume now that $|A\cap B|=k-1$. W.l.o.g. let $a_i=b_i$ for $i\in [k-1]$ and $a_k<b_k$. Then $B+1$ is adjacent to $A$ and $B$. So, we have:

\begin{observation}\label{rem:AB}
Let $A,B\in [n]_2^k$. If $|A\cap B|=k-1$ then $\text{dist}(A,B)=2$ and if $|A\cap B|=1$ then $\text{dist}(A,B)\in\{2,3\}$.
\end{observation}

Let $A,B \in [n]_2^k$. In order to compute the distance between vertices $A$ and $B$, we consider the subsets of vertices in the cycle $C_n$ with vertex set $[n]$ induced by the elements in $A \cup B$. \\

\noindent Let $X=A\cup B\subseteq [n]$. We denote $\mathcal X$ to the family of connected components of the graph induced by $X$ in the $n$-cycle $C_n$. In the same way, $\overline{\mathcal X}$ is the family of connected components of the graph induced by $[n]\setminus X$ in $C_n$. Let $\mathcal P=\{C\in\overline{\mathcal X}:\ |C|\text{ is even}\}$ and $\mathcal I=\{C\in\overline{\mathcal X}:\ |C|\text{ is odd}\}$. From these definitions, we have the following simple observations and Lemma \ref{th:s2dist2}.
\begin{observation}
\label{observation3}
Let $A,B \in [n]_2^k$ with $A \cap B \neq \emptyset$. Let $X = A \cup B$. Then,
\begin{enumerate}
\item 
$|\mathcal{X}|=\left|\overline{\mathcal X}\right|$.
\item If $\left|\overline{\mathcal X}\right|\geq k$ then,  $dist(A,B)=2$.
\item If $|C|\leq 2$ for every $C\in\mathcal{X}$ then,  $dist(A,B)=2$.
\end{enumerate}
\end{observation}

\begin{lemma}\label{th:s2dist2}
Let $A,B\in [n]_2^k$ with $A\cap B\neq\emptyset$ and $X=A\cup B$. Then, $dist(A,B)=2$ if and only if $\frac{1}{2}|\mathcal I|+\frac{1}{2}\sum\limits_{C\in\overline{\mathcal X}}|V(C)|\geq k$.
\end{lemma}

\begin{proof}
Let $C\in\mathcal P$. Note that the maximum cardinality of a $2$-stable subset of $V(C)$ is $\frac{1}{2}|V(C)|$. If $C\in\mathcal I$, the maximum cardinality of a $2$-stable subset of $V(C)$ is $\frac{1}{2}(|V(C)|+1)$. For each $C\in\overline{\mathcal X}$, let $C_2$ be a maximum cardinality $2$-stable subset of $C$. It is not hard to see that $Y=\bigcup\limits_{C\in\overline{\mathcal X}}C_2$ is a $2$-stable subset of $[n]\setminus (A\cup B)$ with cardinality $\sum\limits_{C\in\mathcal P}\frac{1}{2}|V(C)|+\sum\limits_{C\in\mathcal I}\frac{1}{2}(|V(C)|+1)=\frac{1}{2}|\mathcal I|+\frac{1}{2}\sum\limits_{C\in\overline{\mathcal X}}|V(C)|$.

Hence, if $\frac{1}{2}|\mathcal I|+\frac{1}{2}\sum\limits_{C\in\overline{\mathcal X}}|V(C)|\geq k$, there exists a $2$-stable $k$-subset $Y'$ of $[n]\setminus X$. Therefore, $Y'$ is adjacent to $A$ and $B$ in $\operatorname{SG}(n,k)$ and $dist(A,B)=2$.

In order to prove the converse, it is enough to note that the maximum cardinality of a $2$-stable subset in $[n]\setminus X$ is $\frac{1}{2}|\mathcal I|+\frac{1}{2}\sum\limits_{C\in\overline{\mathcal X}}|V(C)|$. Then, if $\frac{1}{2}|\mathcal I|+\frac{1}{2}\sum\limits_{C\in\overline{\mathcal X}}|V(C)|<k$, the distance between $A$ and $B$ in $\operatorname{SG}(n,k)$ is at least $3$.
\end{proof}

Notice that if $2k+ 2\leq n\leq 4k-3$ (i.e. if $2 \leq r \leq 2k-3$ and $n=2k+r$) then $k \geq 3$ and $\operatorname{D}(\operatorname{SG}(n,k)) \geq3 $ since vertices $A=\{1,4,6,\dots,2k\}$ and $B=\{1,5,7,\dots,2k+1\}$ are at distance $3$ in $\operatorname{SG}(n,k)$. In fact, observe that $[n]\setminus (A\cup B)=\{2,3\}\cup\{2k+2,\dots,n\}$ and then there is no $2$-stable $k$-subset in $[n]\setminus (A\cup B)$, i.e. there is no vertex of $\operatorname{SG}(n,k)$, adjacent to $A$ and $B$. Finally, notice that the vertices $A$, $\{3,5,7\dots,2k+1\}$, $\{2,4,6,\dots,2k\}$ and $B$ induce a $P_4$ in $\operatorname{SG}(n,k)$.\\
On the other hand, if $n\geq 4k-2$, then Lemma \ref{th:s2dist2} is enough to assure that $\operatorname{D}(\operatorname{SG}(n,k))=2$.

\begin{theorem}\label{cor:s2diam2}
Let $n,k$ and $r$ be positive integers, with $k \geq 2$ and $n = 2k+r$. $\operatorname{D}(\operatorname{SG}(n,k))=2$ if and only if $r \geq 2k-2$.
\end{theorem}
\begin{proof}
By the preceding discussion, it only remains to show that if $r\geq k-2$, then $\operatorname{D}(\operatorname{SG}(n,k))=2$.
Let $A,B \in [n]_2^k$ such that $A \cup B \neq \emptyset$. If $r \geq 2k-2$ then $n \geq 4k-2$ and thus, $|[n]\setminus X|\geq2k+2k-2-(2k-1)=2k-1$. Hence, 
\[
\frac{1}{2}|\mathcal I|+\frac{1}{2}\sum\limits_{C\in\overline{\mathcal X}}|V(C)|\geq k.
\]
Therefore, by Lemma \ref{th:s2dist2} the result holds.
\end{proof}

Let $A,B \in [n]_2^k$, with $A \neq B$ and $A \cap B \neq \emptyset$, and let $X = A \cup B$. In what follows, we will study the structure of $\mathcal{X}$ and $\overline{\mathcal{X}}$ in order to construct a path between vertices $A$ and $B$ of length as short as possible. \\

Notice that a connected component $C$ in $\mathcal{X}$ is either a single element in $A\cap B$, or it alternates between vertices of $A$ and $B$. Furthermore, if $A\neq B$ and $A\cap B\neq \emptyset$, then there is at least one component in $\mathcal{X}$ made from a single element in $A\cap B$, and at least one component not containing elements in $A\cap B$. Consider the following example:
\begin{example}
\label{ex1}
Let $A,B \in [20]_2^7$, where $A = \{2,8,10,12,15,18,20\}$ and $B = \{1,6,8,10,12,14,17\}$. Thus, we have that $A \cap B = \{8,10,12\}$; $X = A \cup B = \{1,2,6,8,10,12,14,15,17,18,20\}$; $A \setminus B = \{2,15,18,20\}$; $B \setminus A = \{1,6,14,17\}$; $\mathcal{X} = \{\{20,1,2\},\{6\},\{8\},\{10\},\{12\},\{14,15\},\{17,18\}\}$, and $\overline{\mathcal{X}} = \{\{3,4,5\},\{7\},\{9\},\{11\},\{13\},\{16\},\{19\}\}$.
\end{example}

Now, we want to construct two sets $A^*, B^* \in [n]_2$ such that $A^* \subset \overline{A}$, $|A^*| \geq k$, $B \setminus A \subset A^*$,  $B^* \subset \overline{B}$, $|B^*| \geq k$, and $A \setminus B \subset B^*$. Once sets $A^*$ and $B^*$ are constructed, we want to find two subsets $A' \subseteq A^*$ and $B' \subseteq B^*$ such that $A \cap A' = B \cap B' = \emptyset$ and $|A'| = |B'| = k$. Furthermore, we want $A' \cap B' = \emptyset$ or, if this cannot be achieved, 
we want the intersection to be as small as possible.\\

Looking at Example \ref{ex1}, we start with $\{1,6,14,17\}\subset A^*$ and $\{2,15,18,20\}\subset B^*$. Notice that $5,7,13$ and $16$ cannot be in $A^*$, as we want it to be in $[n]_2$. This happens because $\ell \in \{6,12,17\}$ is a vertex in components of $\mathcal{X}$ and $\ell -1$ or $\ell +1$ is a vertex in components of $\bar{\mathcal{X}}$. Similarly $3,16,$ and $19$ cannot be in $B^*$. With that in mind, we can have $3\in A^*$ and $4\in B^*$, but then $5$ cannot be in either; $7$ and $13$ can only be in $B^*$; $16$ cannot be in either; $19$ can be in $A^*$. Thus far we have $\{1,3,6,14,17,19\}\subset A^*$, and $\{2,4,7,13,15,18,20\}\subset B^*$. As the elements in $A\cap B$ are in neither $A^*$ nor $B^*$, we have no restrictions for $9$ and $11$. This means that we can have them in $A^*$ or in $B^*$ (even in both, if it was necessary 
for both of them to have at least $k$ elements). As right now $B^*$ already has $k=7$ elements, we can have $9$ and $11$ in $A^*$. Now we have $\{1,3,6,9,11,14,17,19\}\subset A^*$, and $\{2,4,7,13,15,18,20\}\subset B^*$. As both $A^*$ and $B^*$ are in $[n]_2$ and each of them has at least $7$ elements, we stop adding elements to them. Finally, we can obtain $A'$ by eliminating any element from $A^*$, say $A'=\{1,3,6,9,14,17,19\}$, and we can have $B'=B^*$. Because of how we build them, the vertices $A,A',B',$ and $B$ induce a $P_4$ in  $\operatorname{SG}(20,7)$, which means that $\text{dist}(A,B) \leq 3$.

\subsubsection*{Construction of sets $A^*$ and $B^*$ and an upper bound for $\text{dist}(A,B)$}
\label{sec-sets}
In order to construct sets $A^*$ and $B^*$ corresponding to Example \ref{ex1}, we care particularly about the length of the connected components in $\overline{\mathcal{X}}$ and their relation with the end-vertices of components of $\mathcal{X}$.
In particular, being able to use all the elements in a connected component $C\in \overline{\mathcal{X}}$ depends on the parity of $|C|$ and on the end-vertices of $C$. From now on, by Theorem \ref{cor:s2diam2}, we assume that $n = 2k + r$ with $2 \leq r \leq 2k-3$, that is, $2k+2 \leq n \leq 4k-3$ and we assume that $\text{dist}(A,B) \geq 3$.\\

We say that $\ell\in[n]$ is an \emph{end} if $\ell\in X$ and $|\{\ell-1,\ell+1\}\cap \overline{X}|\geq 1$. This is, 
$\ell$ is an endpoint if it is in $X$ and at least one of its neighbors in the cycle is in $\overline{X}$.
Furthermore, we say that $\ell$ is an \emph{$A$-end} if $\ell\in A\setminus B$, a \emph{$B$-end} if $\ell\in B\setminus A$
and an \emph{$H$-end} if $\ell\in A\cap B$. Finally, let $e(A),e(B)$ and $e(H)$ be the sets of $A$-ends, $B$-ends and $H$-ends
respectively. Notice that $e(H)=A\cap B$, as every vertex in $A\cap B$ must be an end. Let $h=|e(H)|$. Finally, notice that if $A\neq B$ and $A\cap B\neq \emptyset$ then
$e(H)\neq \emptyset$ and $e(A)\cup e(B)\neq \emptyset$ (actually, neither $e(A)$ nor $e(B)$ are empty). In Example \ref{ex1}, $e(A)=\{2,15,18,20\}$, $e(B)=\{6,14,17\}$, and $e(H)=\{8,10,12\}$. 
Notice that $|e(A)|=4$ and $|e(B)|=3$, which means that, in general, $e(A)$ and $e(B)$ 
are not necessarily equal. To obtain the necessary relation between $e(A)$ and $e(B)$, it is helpful to study the structure of the components in $\mathcal{X}$.\\

We say that a connected component $C\in\mathcal{X}$ is 
an $A$-component if $|C\cap e(A)|\geq 1$ and $|C\cap e(B)|=0$. Notice that this 
can happen in two different ways, either $|C|\geq 3$ and $|e(A)\cap C|=2$ or $|C|=1$ and
$e(A)\cap C=C$.
Similarly, we say that a connected component $C\in \mathcal{X}$ is a $B$-component if $|C\cap e(B)|\geq 1$ and $|C\cap e(A)|=0$.
If $C$ is neither an $A$-component nor a $B$-component, we say that $C$ is an $H$-component.
Notice that if $C$ is an $H$-component, then $|C|=1$ if and only if $C\subset A\cap B$. In this case, we say that $C$ is an $H'$-component, otherwise, we say that $C$ is an $H''$-component.
In Example \ref{ex1}, $\{1,2,20\}$ is an $A$-component, $\{6\}$ is a $B$-component,
and $\{8\}$, $\{10\}$, $\{12\}$, $\{14,15\}$, and $\{17,18\}$ are $H$-components,
where $\{8\}$, $\{10\}$ and $\{12\}$ are $H'$-components, and $\{14,15\}$ and $\{17,18\}$ are $H''$-components.
By $n(A)$, $n(B)$, $n(H)$, $n(H')$ and $n(H'')$ we denote the number of $A$-components, $B$-components $H$-components, $H'$-components and $H''$-components
respectively.
Notice that in Example \ref{ex1}, $n(A)=n(B)$, which is actually true in general as it is shown in Lemma \ref{leman(A)=n(B)}.

\begin{lemma}\label{leman(A)=n(B)}
The number of $A$-components equals the number of $B$-components.
\end{lemma}

\begin{proof}
Consider the components in $\mathcal{X}$. If a component $C\in \mathcal{X}$ is an $H$-component, then $C$ has the same number of elements in $A$ and in $B$, i.e. $|C\cap A|=|C\cap B|$.
If $C$ is an $A$-component, then $C$ has one extra element in $A$, i.e. $|C\cap A|=|C\cap B|+1$.
Similarly if $C$ is a $B$-component we get $|C\cap A|+1=|C\cap B|$.  As $|A|=|B| = k$, we get that  the number of $A$-components coincides with the number of $B$ components.
\end{proof}


Next, we obtain a formula relating $e(A)$ and $e(B)$.
Notice that $|e(A)|$ is equal to twice the number of $A$-components of size at least $3$, plus the number
of $A$-components of size $1$, plus the number of $H''$-components. 
We partition $e(A)$ into $e'(A)$ and $e''(A)$, where $e''(A)$ are the 
elements in $e(A)$ that are in connected components with exactly one element, and $e'(A)$ are the rest. 
In the same way, partition $e(B)$ into $e'(B)$ and $e''(B)$. 

\begin{lemma}\label{lemae(A)=e(B)}
If $A,B\in [n]_2^k$ then, $|e'(A)|+2|e''(A)|=|e'(B)|+2|e''(B)|.$
\end{lemma}

\begin{proof}
Consider the number
$2n(A)+n(H'')$.
Notice that in $n(H'')$ every element in $e'(A)$ which is in an $H''$-component is counted once.
Furthermore, in $2n(A)$, every vertex in $e'(A)$ which is in an $A$-component is counted once,
and every vertex in $e''(A)$ is counted twice. As every element in $e(A)$ is either in an $A$-component or in an $H''$-component, this yields: $2n(A)+n(H'')=|e'(A)|+2|e''(A)|$. In a similar fashion, we get $2n(B)+n(H'')=|e'(B)|+2|e''(B)|$.
Therefore, by applying Lemma \ref{leman(A)=n(B)} we obtain $|e'(A)|+2|e''(A)|=|e'(B)|+2|e''(B)|$.
\end{proof}


Next we turn our focus on the connected components in $\overline{\mathcal{X}}$.
Now, we call \emph{block} to each element of $\overline{\mathcal X}$. Consider the following classifications of a block $[i,j]$:

\begin{itemize}
	\item \emph{Type I}: if $\{i-1,j+1\}\subseteq e(H)$.
	\item \emph{Type II(A)}: if $i-1\in e(H)$ and $j+1\in e(A)$.
	\item \emph{Type II(B)}: if $i-1\in e(H)$ and $j+1\in e(B)$.
	\item \emph{Type III(A)}: if $i-1\in e(A)$  and $j+1\in e(H)$.
	\item \emph{Type III(B)}: if $i-1\in e(B)$ and $j+1\in e(H)$.
	\item \emph{Type IV(A)}: if $\{i-1,j+1\}\subset e(A)$.
	\item \emph{Type IV(B)}: if $\{i-1,j+1\}\subset e(B)$.
	\item \emph{Type IV(H)}: if $[i,j]$ is not of the types above, i.e. if $i-1\in e(A)$ and $j+1\in e(B)$ or vice versa.
\end{itemize}

Note that every block is of exactly one of the types above.
Let $\mathcal T=\{I,$ $II(A),$ $II(B),$ $III(A),$ $III(B),$ $IV(A),$ $IV(B),$ $IV(H)\}$. We define $n(T)$ as the amount of blocks of type $T$, for $T\in\mathcal T$. 
In Example \ref{ex1}, $\{9\}$ and $\{11\}$ are blocks of type $I$, there are no blocks of type $II(A)$, $\{13\}$ is a block of type $II(B)$, there are no blocks of type $III(A)$, $\{7\}$ is a block of type $III(B)$, $\{19\}$ is a block of type $IV(A)$, there are no blocks of type $IV(B)$, and $\{3,4,5\}$ and $\{16\}$ are blocks of type $IV(H)$. 

Notice that if $[i,j]$ is a block, then $i-1$ and $j+1$ are ends in some components of $\mathcal{X}$. In such a case, we say 
that $i-1$ and $j+1$ are \emph{connected} to $[i,j]$.

There is an important difference between type $IV$ and the rest of the types.
If we are trying to form the sets $A^*$ and $B^*$ as was done with
Example \ref{ex1}, and $[i,j]$ is a block of type 
$T\in\{I,II(A),II(B),III(A),III(B)\}$,
then we can use every element in $[i,j]$ because 
we have restrictions in at most one of $\{i,j\}$ for the sets $A^*$ and $B^*$. 
If $[i,j]$ is a block of type $IV(A)$ (or $IV(B)$) of even length, then 
all but one of the elements of $[i,j]$ can be used, because neither $i+1$ nor
$j-1$ can be in $B^*$ ($A^*$ resp.). If $[i,j]$ is a block of type $IV(H)$ of odd length, 
then all but one of the elements of $[i,j]$ can be used, as $i+1$ and $j-1$ cannot both
be in $A^*$ nor both be in $B^*$.
To reflect the fact that for a block $[i,j]$ of type $T\in \{IV(A),IV(B),IV(H)\}$ 
we can only assure the use of $|[i,j]|-1$ elements, we define
\[
m([i,j])=\left\{
\begin{array}[h]{ll}
|[i,j]| & \text{if }[i,j]\text{ is a block of type }T\in\{I,II(A),II(B),III(A),III(B)\};\\
 & \\
|[i,j]|-1 & \text{if }[i,j]\text{ is a block of type }T\in\{IV(A),IV(B),IV(H)\}.	
\end{array}
\right.\]

\begin{lemma}
\label{ecuacion1}
$\sum\limits_{[i,j]\in\overline{\mathcal X}}m([i,j]) \geq n-3k+2h+2$.
\end{lemma}

\begin{proof}
Observe that $\sum_{[i,j]\in\overline{\mathcal X}}m([i,j]) = n - \left|A\cup B\right|-(n(IV(A))+n(IV(B)+n(IV(H))) = \\ n-(2k-h)-(n(IV(A))+n(IV(B)+n(IV(H)))$. Therefore, in order to prove this result, it is enough to show that $(n(IV(A))+n(IV(B)+n(IV(H))) \leq k-h-2$.\\

Let $\ell\in e(H)$.  Notice that $\{\ell-1,\ell+1\}\subset \overline{X}$.
Furthermore, $\ell$ is connected to two blocks of type $T\in\{I,II(A),II(B),III(A),III(B)\}$.
On the other hand, every block of type $I$ is connected to two elements of $e(H)$, while every block of type 
$T\in\{II(A),II(B),III(A),III(B)\}$ is connected to one element of $e(H)$. 
Thus, if we count twice the number of elements in $e(H)$, we count every block of type $I$ twice, and every block of
type $T\in\{II(A),II(B),III(A),III(B)\}$ once. In other words, we get
\begin{align*}
2|e(H)|=&2n(I)+n(II(A))+n(II(B))+n(III(A))+n(III(B))\\
|e(H)|=&n(I)+\frac{n(II(A))+n(II(B))+n(III(A))+n(III(B))}{2}.
\end{align*}
Using that $|e(H)|=h$ we obtain,
\[
h=n(I)+\frac{n(II(A))+n(II(B))+n(III(A))+n(III(B))}{2}.
\]
Furthermore, if $A\neq B$ and $A\cap B\neq \emptyset$,  $n(II(A))+n(II(B))+n(III(A))+n(III(B))\geq 2$,
as there has to be at least one block $[i,j]$ with $i-1\in e(H)$ and $j+1\in e(A)\cup e(B)$
and at least one block $[i,j]$ with $i-1\in e(A)\cup e(B)$ and $j+1\in e(H)$.
Therefore, we get 
\[
n(I)+n(II(A))+n(II(B))+n(III(A))+n(III(B))\geq h+1.
\]
Let us assume that $\left|\overline{\mathcal X}\right|\leq k-1$. Otherwise, by Observation \ref{observation3}.2, we would have that 
$\text{dist}(A,B)=2$ which contradicts the initial assumption $\text{dist}(A,B) \geq 3$.

Therefore, $n(IV(A))+n(IV(B))+n(IV(H))=\left|\overline{\mathcal X}\right|-(n(I)+n(II(A))+n(II(B))+n(III(A))+n(III(B)))\leq k-1-(h+1)=k-h-2$.
\end{proof}

\begin{lemma}\label{lemma:endpoints}
$n(II(A))+n(III(A))+2\ n(IV(A))=n(II(B))+n(III(B))+2\ n(IV(B))$.
\end{lemma}

\begin{proof}
Notice that
every block of type  $T\in\{II(A), III(A), IV(H)\}$ uses one endpoint in $e(A)$, and every block of type
$T\in IV(A)$ uses two endpoints in $e(A)$. Furthermore, every endpoint in $e'(A)$ gets used once, while
every endpoint in $e''(A)$ is used twice. 
This yields
\[
n(II(A))+n(III(A))+n(IV(H))+2\ n(IV(A))=|e'(A)|+2|e''(A)|.
\]
Similarly,
\[
n(II(B))+n(III(B))+n(IV(H))+2\ n(IV(B))=|e'(B)|+2|e''(B)|.
\]
Therefore, Lemma \ref{lemae(A)=e(B)} yields:
\begin{align*}
n(II(A))+n(III(A))+n(IV(H))+2\ n(IV(A))=&    n(II(B))+n(III(B))+n(IV(H))+2\ n(IV(B)).
\end{align*}
Subtracting $n(IV(H))$ from both sides of the equation we obtain the result.
\end{proof}

We are ready now to construct the sets $A^*$ and $B^*$. We want these sets to satisfy:
\begin{itemize}
  \item[\textbf{(C1)}] $A^*,B^*\in [n]_2$;
	\item[\textbf{(C2)}] $A\setminus B\subseteq B^*$, $B\setminus A\subseteq A^*$;
	\item[\textbf{(C3)}] $A\cap A^*=\emptyset$, $B\cap B^*=\emptyset$;
	\item[\textbf{(C4)}] if  $[i,j]$ is a block of type $I,II(A),II(B),III(A)$ or $III(B)$, then $[i,j]\subseteq A^*\cup B^*$;
	\item[\textbf{(C5)}] if  $[i,j]$ is a block is of type $IV(A),IV(B),IV(H)$, every element of $[i,j]$ except at most one belongs to $A^*\cup B^*$.
\end{itemize}
Let us note the relation between the last two items and the definition of $m([i,j])$. For each block $[i,j]$ at least $m([i,j])$ elements belong to $A^*\cup B^*$.

We denote by $Z(i,j)$ the set of vertices in $[i,j]$ at odd distance of $i-1$ in the $n$-cycle in clockwise direction and by $Y(i,j)$ the set of vertices in $[i,j]$ at even distance of $i-1$ in the $n$-cycle in clockwise direction. Besides, let $Z'(i,j)$ be the set of vertices in $[i,j]$ at odd distance of $j+1$ in the $n$-cycle in counterclockwise direction and $Y'(i,j)$ the set of vertices in $[i,j]$ at even distance of $j+1$ in the $n$-cycle in counterclockwise direction.
For instance, if $[4,10]$ is a block, then $Z([4,10])=\{4,6,8,10\}=Z'([4,10])$
and $Y([4,10])=\{5,7,9\}=Y'([4,10])$. On the other hand, if $[4,9]$ is a block, then
$Z([4,9])=\{4,6,8\}=Y'([4,9])$, $Y([4,9])=\{5,7,9\}=Z'([4,9])$.

Note that $Z(i,j)$ and $Z'(i,j)$ are not empty. Besides, these sets are $2$-stable, i.e. they belong to $[n]_2$.
Let us assign elements from the blocks to the sets $A^*$ and $B^*$ by the following rules.

\begin{itemize}
    \item[$R1$] If $[i,j]$ is of type I with at least two elements, include $Z(i,j)$ in $A^*$ and $Y(i,j)$ in $B^*$.
    \item[$R2$] If $[i,j]$ is of type II(A), include $Z'(i,j)$ in $A^*$ and $Y'(i,j)$ in $B^*$.
	\item[$R3$] If $[i,j]$ is of type II(B), include $Z'(i,j)$ in $B^*$ and $Y'(i,j)$ in $A^*$.
	\item[$R4$] If $[i,j]$ is of type III(A), include $Z(i,j)$ in $A^*$ and $Y(i,j)$ in $B^*$.
	\item[$R5$] If $[i,j]$ is of type III(B), include $Z(i,j)$ in $B^*$ and $Y(i,j)$ in $A^*$.
	\item[$R6$] If $[i,j]$ is of type IV(A), include $Z(i,j)$ in $A^*$ and $Y(i,j)\setminus\{j\}$ in $B^*$.
	\item[$R7$] If $[i,j]$ is of type IV(B), include $Z(i,j)$ in $B^*$ and $Y(i,j)\setminus\{j\}$ in $A^*$.
	\item[$R8$] If $[i,j]$ is of type IV(H), with $i-1\in A\setminus B$ and $j+1\in B\setminus A$ include $Z(i,j)\setminus\{j\}$ in $A^*$ and $Y(i,j)$ in $B^*$. If $i-1\in B\setminus A$ and $j+1\in A\setminus B$ include $Z(i,j)\setminus\{j\}$ in $B^*$ and $Y(i,j)$ in $A^*$.
\end{itemize}

Notice that in rules $R1$-$R8$ the elements in blocks of the form $[i,i]$ of type $I$ are not assigned. These elements play a key role that we will mention further. Hence, we define the set $I'$ as the set of such elements, i.e. $I'=\{i\,|\, [i,i]\text{ is a block of type I}\}$.

Consider the sets $A^*$ and $B^*$ constructed following the rules above, and also including to $A^*$ the elements in $B\setminus A$ and assigning to $B^*$ the elements in $A\setminus B$. It is not hard to see that $A^*$ and $B^*$ satisfy:
\begin{itemize}
\item $A^*,B^*\in [n]_2$;
\item  $A\cap A^*=A^*\cap B^*=B\cap B^*=\emptyset$;
\item $|A^* \cap B|=|B^*\cap A|=k-h$;
\item for every block of type $T\in \{II(A),III(A),IV(A)\}$, $|A^*\cap T|\geq 1$;
\item for every block of type $T\in \{II(B),III(B),IV(B)\}$, $|B^*\cap T|\geq 1$; and
\item for every block $[i,j]$ of type $I$ with at least two elements, $|A^*\cap I|\geq 1$ 
and $|B^* \cap I| \geq 1$.
\end{itemize}

From sets $A^*$ and $B^*$ we can construct two vertices $A'\subset A^*$ and $B' \subset B^*$ in $[n]_2^k$ as follows.
\begin{lemma}\label{lemma:htoh-1}
Let $A,B\in [n]_2^k$ with $|A\cap B|=h$. Then, there exist $A',B'\in [n]_2^k$ such that $|A'\cap B'|\leq h-1$ and $A\cap A'=B\cap B'=\emptyset$.
\end{lemma}
\begin{proof}
In order to obtain vertices $A'\subset A^*$ and $B'\subset B^*$ such that $|A'| = |B'| = k$, $A\cap A'=B\cap B'=\emptyset$, and $|A' \cap B'| \leq h-1$, we will use the elements of $I'$.

First, notice that for every
element $i\in A\cap B$ the block of the form $[i+1,j]$ is a block of type $T\in\{I,II(A),II(B)\}$,
thus $h=n(I)+n(II(A))+n(II(B))$. Similarly, for every
element $j\in A\cap B$ the block of the form $[i,j-1]$ is a block of type $T\in\{I,III(A),III(B)\}$,
thus $h=n(I)+n(III(A))+n(III(B))$. Then
\[
2h=2n(I)+n(II(A))+n(II(B))+n(III(A))+n(III(B)).
\]
It follows that $n(II(A))+n(III(A))+n(II(B))+n(III(B))$ is even, so let $s\in\mathbb N$ such that $2s=n(II(A))+n(III(A))+n(II(B))+n(III(B))$.
Then, $2h=2n(I)+2s$, which means $n(I)=h-s$.
Furthermore, assume w.l.o.g. that $n(II(A))+n(III(A))=s+t\geq n(II(B))+n(III(B))=s-t$, for some $t \in \mathbb{N}$. By Lemma
\ref{lemma:endpoints}, 
\[2n(IV(B)) \geq n(II(A))+n(III(A))-n(II(B))-n(III(B))= s+t-(s-t)= 2t.\]
Thus, $IV(B) \geq t$. Hence, $n(II(A))+n(III(A))+n(IV(A))\geq s$ and $n(II(B))+n(III(B))+n(IV(B))\geq s$.
Therefore, from Rules $R2-R7$, we have 
$A^*$ has at least $s$ elements from blocks of types $II(A), III(A)$ and $IV(A)$, and $B^*$
has at least $s$ elements from blocks of types $II(B),III(B)$ and $IV(B)$.

Let $r$ be the amount of blocks $[i,j]$ of type $I$ with at least two elements. Then
from previous remarks and rule $R1$, both $A^*$ and $B^*$ have at least $r$ elements from these blocks. Furthermore, $|I'|=h-s-r$.

Counting again the size of $A^*$, we have
\begin{itemize}
\item $A^*$ has $k-h$ elements from $B$;
\item $A^*$ has at least $s$ elements from blocks of types $II(A),III(A)$ and $IV(A)$;
\item $A^*$ has at least $r$ elements from blocks of types $I$ with at least two elements.
\end{itemize}
Thus, $|A^*|\geq k-h+s+r=k-(h-s-r)$.
This means that if we assign every element in $I'$ to $A^*$, then $|A^*|\geq k$. Similarly, if we assign every element in $I'$ to $B^*$,
then $|B^*|\geq k$.
This also yields $|A^*\cap B^*|\leq h-s-r$.

Notice that there must exist at least one block of type $T$ for some $T\in \{II(A), II(B), III(A), III(B)\}$, as otherwise only blocks of type $I$ would exists, which implies $A=B$.
Hence $s\geq 1$, and $h-s-r\leq h-1$. Therefore, taking $A'\subset A^*$ and $B'\subset B^*$,
with $|A'|=|B'|=k$, we have that $A \cap A' = B \cap B' = \emptyset$ and $|A' \cap B'| \leq |A^* \cap B^*| \leq h-s-r \leq h-1$ which proves the result.
\end{proof}
The following result derives directly from the proof of Lemma \ref{lemma:htoh-1}.

\begin{corollary}
\label{rem:no[ii]}
If there are no blocks $[i,i]$ of type $I$, then $dist(A,B)\leq3$.
\end{corollary}

\begin{proof}
Notice that if there are no blocks $[i,i]$ of type $I$ in the proof of Lemma \ref{lemma:htoh-1} then, $|I'| = h-s-r = 0$ and thus we have that $|A^*| \geq k$, $|B^*| \geq k$ and $A \cap A^* = A^* \cap B^* = B \cap B^* = \emptyset$, and thus, $\text{dist}(A,B) \leq 3$, which implies actually that $\text{dist}(A,B) = 3$.
\end{proof}


\begin{lemma}\label{lemma:dist1+2h}
Let $A,B\in [n]_2^k$ with $|A\cap B|=h$. Then, $dist(A,B)\leq 1+2h$.
\end{lemma}
\begin{proof}
By applying Lemma \ref{lemma:htoh-1} $h$ times, we obtain two vertices $A^{(h)},B^{(h)}\in [n]_2^k$, with $A^{(h)}\cap B^{(h)}=\emptyset$. Hence, $dist(A,B)\leq 1+2h$.
\end{proof}

\begin{corollary}\label{cor:2+2h}
Let $A,B,Y\in [n]_2^k$ with $|A\cap Y| = h'$, $|Y\cap B| = h''$, and let $h^* = h'+h''$ with $h^* \geq 2$. Then, $dist(A,B)\leq 2+2h^*$.
\end{corollary}
\begin{proof}
The result follows by applying Lemma \ref{lemma:dist1+2h} to bound $dist(A,Y)$ and $dist(Y,B)$.
\end{proof}

Lemmas \ref{ecuacion1} and \ref{lemma:dist1+2h} and Corollaries  \ref{rem:no[ii]} and \ref{cor:2+2h} will be used to compute the diameter of $\text{SG}(n,k)$ when $2k+2 \leq n \leq 4k+3$ in the next subsections.

\subsection{Case $3k-2 \leq n \leq 4k-3$}
Let $n = 2k+r$ with $k > 2$ and $k-2 \leq r \leq 2k-3$. Let us consider the construction of sets $A^*$ and $B^*$ given in Section \ref{sec-sets}. Let remark that in the proof of Lemma \ref{lemma:htoh-1}, we may not need to assign every element in $I'$ to both $A^*$ and $B^*$. Note that following the rules $R2$-$R8$, from each block $[i,j]$ of type $II$, $III$ or $IV$, we have included at least $m([i,j])$ elements in $A^*\cup B^*$ such that $A^*\cap B^*=\emptyset$ and $A^*,B^*\in [n]_2$.  
If we do not assign the $h-s-r$ elements in $I'$, by Lemma \ref{ecuacion1}, we have assigned at least $n-3k+2h+2-(h-s-r)$ elements from blocks to $A^*\cup B^*$, $k-h$ elements from $A$ and $k-h$ elements from $B$. This means that before assigning the $h-s-r$ elements in $I'$, we have assigned $n-k+2-(h-s-r)$ elements to $A^*\cup B^*$. Hence,  if $n-k+2$ is large enough, we may be able to assign the elements in such blocks maintaining $A^*\cap B^*=\emptyset$.

Assume $n\geq 3k-2$. Then, before assigning the elements in $I'$, we have assigned 
at least $n-k+2-(h-s-r)\geq 2k-h+s+r$ to $A^*\cup B^*$, at least $k-h+s+r$ elements to $A^*$, and
at least $k-h+s+r$ elements to $B^*$. Let $0 \leq a \leq h-s-r$ and assume that we have assigned $k-h+s+r+a$ elements to $A^*$.
This means that we assigned at least $2k-h+s+r-(k-h+s+r+a)=k-a$ elements to $B^*$.
If $h-s-r-a>0$, assign that many elements from $I'$ to $A^*$ and the remaining $a$ elements to $B^*$,
otherwise, if  $h-s-r-a = 0$, assign every element in $I'$ to $B^*$. Then $|A^*|\geq k$, $|B^*|\geq k$,
and $A^*\cap B^*=\emptyset$. Let $A'\subset A^*$ and $B'\subset B^*$, such that 
$|A'|=|B'|=k$. Then we have $A\cap A'=A'\cap B'=B'\cap B=\emptyset$, which means 
that $dist(A,B)\leq 3$. As by hypothesis, $n < 4k-2$ then, by Theorem \ref{cor:s2diam2}, we deduce that $\text{dist}(A,B) \geq 3$. 

\begin{theorem}\label{th:diam3}
Let $n = 2k+r$ with $k>2$ and $k-2 \leq r \leq 2k-3$. Then, $\operatorname{D}(\operatorname{SG}(n,k))=3$.
\end{theorem}

Notice that in Rules $R1-R8$ we have not assigned elements in blocks of type $IV(H)$ of the form $[i,i]$.
\begin{observation}\label{rem:R1-R8}
Let $k> 2$ and $3k-2\leq n\leq 4k-3$. If two vertices $A,B$ are at distance $3$, there exist two vertices $A',B'$ constructed following the rules $R1$-$R8$ such that $\{A,A',B',B\}$ induce a $P_4$ in $\operatorname{SG}(n,k)$. Besides, if $[i,i]$ is a block of type $IV(H)$, $i\notin A'\cup B'$.
\end{observation}



\subsection{Case $2k+2\leq n\leq 3k-3$}
In this section, we show that $\operatorname{D}(\operatorname{SG}(2k+r,k)) \leq k-r+1$ when $2 \leq r \leq k-3$, or equivalently, $\operatorname{D}(\operatorname{SG}(3k-2-m,k))\leq 3+m$ for $1\leq m\leq k-4$.
To do this, given two vertices $A,B \in [n]_2^k$, we use two different operations that yield
sets $\tilde{A}$ and $\tilde{B}$ in $[n+1]^k_2$. We apply the operations successively
until we obtain two vertices $A^p,B^p\in [n+p]^k_2$ with $dist(A^p,B^p)\leq 3$ in $\text{SG}(n+p,k)$.
If $dist(A^p,B^p)=2$ in $\text{SG}(n+p,k)$, we obtain a vertex $Y \in [n]_2^k$ such that $dist(A,Y)+dist(B,Y)\leq m+3$.
If $dist(A^p,B^p)=3$ in $\text{SG}(n+p,k)$, we obtain two vertices $A',B' \in [n]_2^k$ using Rules $R1$-$R8$ such that $A\cap A'=B\cap B'=\emptyset$,
and $dist(A',B')\leq 1+m$.\\

Now we can begin describing the operations that yield sets in $[n+1]^k_2$. The first 
operation will work by adding an element in a component $C\in \mathcal{X}$, with $|C|\geq 3$;
the second operation will work by adding an element to a block $[t,t]$ of type $I$.
Because we are going to be talking about distances in Schrijver graphs with different values
of $n$,  we denote $dist_n(A,B)$ the distance between $A$ and $B$ in $\operatorname{SG}(n,k)$.\\

Let $A,B\in [n]_2^k$ such that $dist_n(A,B)\geq 3$ and $X=A\cup B$. From item $(iii)$ in Observation \ref{observation3}, there exist $C\in\mathcal{X}$ such that $|C|\geq 3$. Consider $C=a_i\ b_j\ a_{i+1}\ldots$ (first case) or $C=b_j\ a_i\ b_{j+1}\ldots$ (second case).
To obtain sets in $[n+1]_2^k$ we will add an extra element between the second and third elements
in $C$, i.e. between $b_j$ and $a_{i+1}$ in the first case, and  between $a_i$ and $b_{j+1}$ in the second case. 
Hence, we assign to $A$ and $B$ (in any case) the following sets in $[n+1]$, $A^+=\{a_1,\ldots,a_i,a_{i+1}+1,\ldots,a_k+1\}$ and $B^+=\{b_1,\ldots,b_{j},b_{j+1}+1,\ldots,b_k+1\}$.
Notice that $|A^+|=|B^+|=k$, as we did not increase the amount of elements. Furthermore,
the sets are $2$-stable, because $A$ and $B$ are $2$-stable ($a_1$ and $a_{k}+1$ cannot be consecutive, because that would imply that $a_1=1$ and $a_k+1=n+1$,
and $a_k=n$). Therefore, $A^+,B^+\in [n+1]_2^k$.

By adding this new element, we formed a new block.
Observe that if $C=a_i\ b_j\ a_{i+1}\ldots$, then $a_{i+1}\notin A^+\cup B^+$, and if $C=b_j\ a_i\ b_{j+1}$, then $b_{j+1}\notin A^+\cup B^+$. Thus, in the first case, $[a_{i+1},a_{i+1}]$ is a block of type $IV(H)$ in $\overline{\mathcal{X}}$ with $X=A^+\cup B^+$. Analogously, in the second case, $[b_{j+1},b_{j+1}]$ is a block of type $IV(H)$ in $\overline{\mathcal{X}}$.

Concerning the inverse operation of operation $+$, for $Y\in [n+1]^k_2$, we define a set $Y^-=\{y^-_1,\dots,y^-_k\}\in [n]^k$ by deleting the element that we added. In other words, considering $u=a_{i+1}$ if we are in the first case
and $u=b_{j+1}$ if we are in the second case, we have that $y^-_r=y_r$ if $y_r<u$ and $y^-_r=y_r-1$ if $y_r\geq u$.\\


The following remark is straightforward from the previous definitions.

\begin{observation}\label{rem:empty}
If $Y\cap A^{+}\cap B^{+}=\emptyset$ then $Y^-\cap A\cap B=\emptyset$.
\end{observation}

Notice that $Y^-$ is not $2$-stable if and only if $u-1,u+1\in Y$ or if $u-2,u\in Y$.
But if $X=A^+ \cup B^+$, then $\{u-2,u-1\}$ is a connected component of $\mathcal{X}$, and 
$u+1$ is the first element (in clockwise direction) in a connected component in $\mathcal{X}$.
\begin{observation}\label{rem:Y-stab}
Let $Y\in [n+1]^k_2$ and suppose that for every element $v\in Y\cap (A^+\cup B^+)$, $v$ is not the first element of a connected component $C\in \mathcal{X}$. Then $\{u-2,u+1\}\cap Y=\emptyset$ and $Y^-\in[n]^k_2$.
\end{observation}
Observation \ref{rem:Y-stab} is stated in such a convoluted way
to make easier the proof of the main theorem in this section, 
which uses successive applications of  $+$, together with a second operation
which is defined later in this section.

As Observation \ref{rem:Y-stab} assures that $\{u-2,u+1\}\cap Y=\emptyset$, we can study the relation between $Y^-\cap A$
and $Y\cap A^+$, and similarly with $B$, to obtain the following.
\begin{observation}\label{rem:cap+1}
Let $Y\in [n+1]^k_2$ such that $u-2$ and $u+1$ are not in $Y$. If $Y^-$ is defined as above, then $Y^-\in [n]^k_2$ and $|Y\cap A^{+}|+|Y\cap B^{+}|=|Y^-\cap A|+|Y^-\cap B|$ if $u\notin Y$, or $|Y\cap A^{+}|+|Y\cap B^{+}|+1=|Y^-\cap A|+|Y^-\cap B|$ if $u\in Y$. Furthermore, if no element $v\in Y \cap (A^+\cup B^+)$ is the first element in a connected component $C$ of $A^+\cup B^+$, then no element $v\in Y^-\cap (A\cup B)$ is the first element in a connected component $C$ of $A\cup B$.
\end{observation}
Notice that if $u-2,u\in A$, and $Y\cap A^+=\emptyset$ then $u-2,u+1\not\in Y$ and $Y^-\cap A=\emptyset$. On the other hand, if $u-1\in A$, $Y\cap A^+=\emptyset$ and
$u\not\in Y$, then $u-1,u\not\in Y$ and $Y^- \cap A=\emptyset$. This yields the following.
\begin{observation}\label{rem:AY+empty}
Let $Y\in [n+1]^k_2$. If $Y\cap A^+=\emptyset$ and $u\not\in Y$, then $Y^-\in[n]_2^k$ and $A\cap Y^- =\emptyset$.
\end{observation}
Finally, let $Y_1,Y_2\in[n+1]_2^k$
and consider $Y_1^-\cap Y_2^-$. If $\{u-1,u\}\not\subset Y_1\cup Y_2$, then
$y\in Y_1\cap Y_2$ if and only if $y^-\in Y_1^- \cap Y_2^-$ (where $y^-$ is as in the definition
of $Y^-$). Hence, we have the following
\begin{observation}\label{rem:Y1-capY2-}
Let $Y_1,Y_2\in[n+1]_2^k$.
If $u\not\in Y_1\cup Y_2$, then $|Y_1^-\cap Y_2^-|=|Y_1 \cap Y_2|$.
\end{observation} 

\medskip
We are ready to introduce the second operation.
Notice that if $dist_n(A,B)\geq 4$, Corollary \ref{rem:no[ii]} implies that there exists a block $[t,t]$ of type $I$. 
The operation will work by adding an extra element at position $t+1$, thus increasing 
the size of the block.
To be more precise, we define an operation on $A$ and $B$, denoted $\uparrow$, by assigning the following two sets $A^{\uparrow}=\{a^{\uparrow}_1,\ldots,a^{\uparrow}_k\}$ and $B^{\uparrow}=\{b^{\uparrow}_1,\ldots,b^{\uparrow}_k\}$ in $[n+1]_2^k$ as follows:

\[a^{\uparrow}_i=\left\{\begin{array}{ll}
    a_i & \text{if }a_i\leq t-1; \\
    a_i+1 & \text{if }a_i\geq t+1.
\end{array}
\right.\quad\text{and}\quad b^{\uparrow}_i=\left\{\begin{array}{ll}
    b_i & \text{if }b_i\leq t-1; \\
    b_i+1 & \text{if }b_i\geq t+1.
\end{array}
\right.\]

Note that, if $X=A^{\uparrow}\cup B^{\uparrow}$ ($X\subseteq [n+1]$), then $[t,t+1]$ is a block of type $I$ in $\overline{\mathcal{X}}$.

We define now the inverse operation of operation $\uparrow$. Given $Y\in [n+1]^k_2$, we define a set $Y^{\downarrow}=\{y^{\downarrow}_1,\dots,y^{\downarrow}_k\}\in [n]^k$ as follows:

\[y^{\downarrow}_r=\left\{\begin{array}{ll}
    y_r & \text{if }y_r\leq t; \\
    y_r-1 & \text{if }y_r\geq t+1.
\end{array}
\right.\]

As with the operation $+$, we care about when $Y^{\downarrow}$ is $2$-stable, and about
the relation between the intersections of $Y$ with $A^{\uparrow}$ and $B^{\uparrow}$, and
the intersection of $Y^{\downarrow}$ with $A$ and $B$.
Notice that if $Y\cap[A^{\uparrow}\cup B^{\uparrow}]=\emptyset$, $Y^{\downarrow}\in [n]^k_2$. Moreover, we have the following result.

\begin{lemma}\label{lemma:samecap}
Let $A,B\in [n]^k_2$ such that there exists a block $[t,t]$ of type $I$ in $\mathcal X$ with $X=A\cup B$.
Consider $A^{\uparrow}$ and $B^{\uparrow}$ defined as above. Let $Y\in [n+1]^k_2$ such that $Y\cap A^{\uparrow}\cap B^{\uparrow}=\emptyset$ and $Y^{\downarrow}$ defined as above. Then $Y^{\downarrow}\in [n]^k_2$, $Y^{\downarrow}\cap A\cap B=\emptyset$ and $|Y\cap A^{\uparrow}|+|Y\cap B^{\uparrow}|=|Y^{\downarrow}\cap A|+|Y^{\downarrow}\cap B|$.
\end{lemma}
\begin{proof}
Consider $[t,t+1]$ the block of type $I$ defined as above. Since $Y\in [n+1]^k_2$, $|Y\cap\{t,t+1\}|\leq1$. In addition, from the fact that $Y\cap A^{\uparrow}\cap B^{\uparrow}=\emptyset$, $t-1$ and $t+2$ are not in $Y$. Thus $t-1$ and $t+1$ are not in $Y^{\downarrow}$. Therefore, $Y^{\downarrow}$ is in $[n]^k_2$, $Y^{\downarrow}\cap A\cap B=\emptyset$ and $|Y\cap A^{\uparrow}|+|Y\cap B^{\uparrow}|=|Y^{\downarrow}\cap A|+|Y^{\downarrow}\cap B|$.
\end{proof}

In particular, the following result follows immediately from Lemma \ref{lemma:samecap} if $Y\cap (A^{\uparrow}\cup B^{\uparrow})=\emptyset$.

\begin{corollary}\label{cor:dist2}
Let $[t,t]$ be a block of type $I$ in $\mathcal X$ with $X=A\cup B$. If $dist_{n+1}(A^{\uparrow},B^{\uparrow})=2$ then $dist_{n}(A,B)=2$.
\end{corollary}

Let $Y_1,Y_2\in[n+1]_2^k$
and consider $Y_1^{\downarrow}\cap Y_2^{\downarrow}$. If $\{t,t+1\}\not\subset Y_1\cup Y_2$, then
$y\in Y_1\cap Y_2$ if and only if $y^{\downarrow}\in Y_1^{\downarrow}\cap Y_2^{\downarrow}$
(where $y^{\downarrow}$ is as in the definition of $Y^{\downarrow}$).
If $\{t,t+1\}\subset Y_1\cup Y_2$, then $t\in Y_1^{\downarrow}\cap Y_2^{\downarrow}$.
Hence we have the following.
\begin{observation}\label{rem:Y1downcapY2down}
If $Y_1,Y_2\in[n+1]_2^k$,
then $|Y_1^{\downarrow}\cap Y_2^{\downarrow}|\leq |Y_1\cap Y_2|+1$.
\end{observation}


Starting from $A^0 = A$ and $B^0 = B$ and by applying repeatedly  the operations $+$ and $\uparrow$, we are able to construct two vertices $A^p, B^p \in [n+p]_2^k$, with $p \leq m$, such that $\text{dist}_{n+p}(A^p,B^p) \leq 3$. Then, by applying repeatedly the operations $-$ and $\downarrow$, we obtain the following result. 

\begin{theorem}\label{th:boundm+3}
Let $n=3k-2-m$ with $1\leq m\leq k-4$. Then, $\operatorname{D}(\operatorname{SG}(n,k))\leq m+3$.
\end{theorem}

\begin{proof}
Let $A,B\in [n]_2^k$ and assume that $dist_n(A,B)\geq4$.
Corollary \ref{rem:no[ii]} and Observation \ref{observation3}.3 imply that both $+$ and $\uparrow$
can be applied. 
Let $A^0=A$, $B^0=B$, and for $1\leq \ell \leq p$ let
\[
A^\ell=\left\{\begin{array}{ll}
    \left(A^{\ell-1}\right)^+ & \text{if $\ell$ is odd}; \\
    \left(A^{\ell-1}\right)^{\uparrow} & \text{if $\ell$ is even}; \\
\end{array}\right.\quad\text{and}\quad B^\ell=\left\{\begin{array}{ll}
    \left(B^{\ell-1}\right)^+ & \text{if $\ell$ is odd}; \\
    \left(B^{\ell-1}\right)^{\uparrow} & \text{if $\ell$ is even}, \\
\end{array}\right.\]
where $p$ is the smallest integer such that $dist_{n+p}(A^p,B^p)\leq 3$.
Theorem \ref{th:diam3} implies that $p\leq m$.

Let $X^\ell=A^\ell \cup B^\ell$, and $\mathcal{X}^\ell$ be the family of connected 
components of $X^\ell$. 
We divide the proof in two cases: $dist_{n+p}(A^p,B^p)=2$ and $dist_{n+p}(A^p,B^p)=3$.

If $dist_{n+p}(A^p,B^p)=2$, there exists $Y\in [n+p]^k_2$ such that $Y\cap (A^p\cup B^p)=\emptyset$. Let $Y^p=Y$, and for $0\leq \ell \leq p-1$ let 
\[
Y^\ell=\left\{\begin{array}{ll}
    \left(Y^{\ell+1}\right)^- & \text{if $\ell$ is even}; \\
    \left(Y^{\ell+1}\right)^{\downarrow} & \text{if $\ell$ is odd}. \\
\end{array}\right.\]

From Corollary \ref{cor:dist2}, it follows that the last operation applied to $A$ and $B$ is $+$. Thus, $p$ is odd and the sets $A^p$ and $B^p$ are obtained after applying $\frac{p+1}{2}$ operations $+$ and $\frac{p-1}{2}$ operations $\uparrow$. Furthermore, Observations \ref{rem:empty},
\ref{rem:Y-stab} and \ref{rem:cap+1} imply $Y^{p-1}\cap A^{p-1}\cap B^{p-1}=\emptyset$, $Y^{p-1}\in [n+p-1]^k_2$ and $|Y^{p-1}\cap A^{p-1}|+|Y^{p-1}\cap B^{p-1}|\leq 1$. Furthermore, if 
$v\in Y^{p-1}\cap (A^{p-1}\cup B^{p-1})$, then $v$ is not the first
element in a connected component
of $X^{p-1}$.

Now, the sets $A^{p-1},B^{p-1}$ are obtained from operation $\uparrow$. Since $Y^{p-1}\cap A^{p-1}\cap B^{p-1}=\emptyset$, from Lemma \ref{lemma:samecap}, $Y^{p-2}\in [n+p-2]^k_2$ and $|Y^{p-2}\cap A^{p-2}|+|Y^{p-2}\cap B^{p-2}|=|Y^{p-1}\cap A^{p-1}|+|Y^{p-1}\cap B^{p-1}|\leq1$. Furthermore,
if $v\in Y^{p-2}\cap (A^{p-2}\cup B^{p-2})$, then $v$ is not the first
element in a connected component
of $X^{p-2}$.

Applying this reasoning repeatedly, we get that
 for $q\geq 1$, $Y^{p-2q-1}\in [n+p-2q-1]^k_2$ and $|Y^{p-2q-1}\cap A^{p-2q-1}|+|Y^{p-2q-1}\cap B^{p-2q-1})|\leq q+1$, and if 
$v\in Y^{p-2q-1}\cap (A^{p-2q-1}\cup B^{p-2q-1})$, then $v$ is not the first
element in a connected component
of $X^{p-2q-1}$.
Similarly, Lemma \ref{lemma:samecap} implies $Y^{p-2q-2}\in [n+p-2q-2]^k_2$, $|Y^{p-2q-2}\cap A^{p-2q-2}|+|Y^{p-2q-2}\cap B^{p-2q-2}|=|Y^{p-2q-1}\cap A^{p-2q-1}|+|Y^{p-2q-1}\cap B^{p-1}|\leq q+1$ and
if $v\in Y^{p-2}\cap (A^{p-2}\cup B^{p-2})$, then $v$ is not the first
element in a connected component
of $X^{p-2q-2}$.

Notice that letting $q=(p-1)/2$ we have $p-2q-1=0$. Then, $|Y^{0}\cap A^{0}|+|Y^{0}\cap B^{0}|\leq 1+(p-1)/2=(p+1)/2$. Then, Corollary \ref{cor:2+2h} implies $dist_n(A,B)\leq 2+2(p+1)/2=p+3\leq m+3$.

\bigskip

Assume now that $dist_{n+p}(A^p,B^p)=3$. From Observation \ref{rem:R1-R8} there exist $ (A^p)',  (B^p)'\in [n+p]^k_2$ constructed following the rules {\textbf{R1-R8}} such that $A^p\cap (A^p)'= (A^p)'\cap (B^p)'=B^p\cap (B^p)'=\emptyset$.

Let $Y_1^p=(A^p)'$ and $Y_2^p=(B^p)'$, and for every $0\leq \ell\leq p-1$, $i\in \{1,2\}$, let 
\[
Y_i^\ell=\left\{\begin{array}{ll}
    \left(Y_i^{\ell+1}\right)^- & \text{if $\ell$ is even}; \\
    \left(Y_i^{\ell+1}\right)^{\downarrow} & \text{if $\ell$ is odd}. \\
\end{array}\right.\]

Let $1\leq q \leq p$. Note that for every element $u$ added through the $+$ operation, $[u,u]$ is a block of type $IV(H)$ and from Observation \ref{rem:R1-R8}, $u\notin (A^p)'\cup\tilde (B^p)'$. Then Observation \ref{rem:AY+empty} and Lemma \ref{lemma:samecap} imply 
that $Y_1^{p-q}\cap A^{p-q}=\emptyset$, $Y_2^{p-q}\cap B^{p-q}=\emptyset$, and $Y_1^{p-q},Y_2^{p-q}\in[n+p-q]_2^k$. This implies that $Y_1^0,Y_2^0\in [n]_2^k$, $A\cap Y_1^0=B\cap Y_2^0=\emptyset$.  

Consider $Y_1^q\cap Y_2^q$. If $q$ is even, then $Y_1^q=\left(Y_1^{q+1}\right) ^-$,
$Y_2^q=\left(Y_2^{q+1}\right) ^-$, and the element added in the corresponding $+$
operation is not in $Y_1^{q+1}\cup Y_2^{q+1}$. Hence, Remark \ref{rem:Y1-capY2-} implies $|Y_1^q\cap Y_2^q|\leq |Y_1^{q+1}\cap Y_2^{q+1}|$. On the other hand, if $q$ is odd, Remark \ref{rem:Y1downcapY2down}
implies $|Y_1^q\cap Y_2^q|=|Y_1^{q+1}\cap Y_2^{q+1}|+1$.

Therefore, $|Y_1^0\cap Y_2^0|$ is bounded by the amount of $\uparrow$ operations applied.
But, as the first operation applied is $+$, the number of $\uparrow$ operations is at most
$p/2$. Thus, $|Y_1^0\cap Y_2^0|\leq p/2$ and Lemma \ref{lemma:dist1+2h} implies
$dist_n(Y_1^0,Y_2^0)\leq 1+2p/2=1+p$. As $A\cap Y_1^0=Y_2^0\cap B=\emptyset$, this means that $dist_n(A,B)\leq 1+p+2=p+3\leq m+3$ and the result follows.
\end{proof}


\subsubsection{A lower bound for $2k+2\leq n\leq 3k-3$}\label{sec:lowerbound}
In this section we will prove that for $n = 2k+r$ with $2\leq r\leq k-3$, the diameter of $\operatorname{SG}(n,k)$ verifies

$$\operatorname{D}(\operatorname{SG}(n,k))\geq 4.$$

\noindent For this, we provide two vertices $A^n_k,B^n_k\in [n]_2^k$ such that $dist_n(A^n_k,B_k^n)\geq4$ as follows.

Let $n=2k+r$ with $2\leq r\leq k-3$ and $t=k-3-r$. Let $A^n_k=\{1,3,5\}\cup\left(\cup_{i=0}^{t}\{7+2i\}\right)\cup\left(\cup_{j=1}^{r-1}\{7+2t+3j\}\right)$ and $B^n_k=\{1,3,6\}\cup\left(\cup_{i=0}^{t}\{8+2i\}\right)\cup\left(\cup_{j=1}^{r-1}\{8+2t+3j\}\right)$.
Note that $A^n_k\cap B^n_k=\{1,3\}$, every block in $\overline{\mathcal X}$ has cardinality $1$ and $|\overline{\mathcal X}|=k-1$. Then, $dist_n(A^n_k,B^n_k)\geq3$. Let $A',B'\in [n]_2^k$ such that $A'$ is a neighbor of $A^n_k$ and $B'$ is a neighbor of $B^n_k$ in $\operatorname{SG}(n,k)$, i.e. $A'\cap A^n_k=B'\cap B^n_k=\emptyset$.

Let $A^*=\cup_{i=0}^{t}\{7+2i-1\}$ and $A^{**}=\cup_{j=1}^{r-1}\{7+2t+3j-2,7+2t+3j-1\}$. Notice that $\overline{A^n_k}=\{2,4,n-1,n\}\cup A^*\cup A^{**}$. Analogously, if $B^*=\cup_{i=0}^{t}\{8+2i-1\}$ and $B^{**}=\cup_{j=1}^{r-1}\{8+2t+3j-2,8+2t+3j-1\}$, $\overline{B^n_k}=\{2,4,5,n\}\cup B^*\cup B^{**}$. Then,
\begin{itemize}
\item $A'\subseteq\overline{A^n_k}=\{2,4,n-1,n\}\cup A^*\cup A^{**}$,
\item $B'\subseteq\overline{B^n_k}=\{2,4,5,n\}\cup B^*\cup B^{**}$.
\end{itemize}

Since $A'$ is a $2$-stable set, $\left|A'\cap\left(A^*\cup A^{**}\right)\right|\leq t+1+r-1=k-3$. In the same way, $\left|B'\cap\left(B^*\cup B^{**}\right)\right|\leq k-3$.

Therefore, $\left|A'\cap\{2,4,n-1,n\}\right|\geq3$ and $\left|B'\cap\{2,4,5,n\}\right|\geq3$. These inequalities imply that $A'\cap B'\neq\emptyset$ and then $A'$ and $B'$ are not adjacent in $\operatorname{SG}(n,k)$.

Thus, $dist_n(A^n_k,B^n_k)\geq4$ if $2k+2\leq n\leq3k-3$.\\

\begin{corollary}
\label{cor:3k-3}
Let $n = 2k+r$ with $2\leq r\leq k-3$. Then $\operatorname{D}(\operatorname{SG}(n,k))\geq4$. In particular, $\operatorname{D}(\operatorname{SG}(3k-3,k))=4$, for $k \geq 5$.
\end{corollary}
\begin{proof}
The lower follows from previous reasoning. From this lower bound and Theorem \ref{th:boundm+3} we have $\operatorname{D}(\operatorname{SG}(3k-3,k))=4$, for $k \geq 5$.
\end{proof}

\section{Case $n=2k+2$}
In order to study the diameter of $\operatorname{SG}(2k+2,k)$, we start by giving a full description
of the graph.

Given a vertex $A\in \operatorname{SG}(2k+2,k)$, let $\overline{\mathcal A}$ denote the family of connected components of the graph induced by $[n]\setminus A$ in $C_n$.
Notice that $\overline{\mathcal{A}}$ must have $k$ connected components. This implies that
it can have either one connected component of order $3$ and the rest of order $1$, or two connected
components of order $2$ and the rest of order $1$. Furthermore, in the later case,
the shortest path in $C_n$ between the components of order $2$ of $\overline{\mathcal{A}}$ (i.e. the minimum length of the two intervals on $C_n$ separating these two components) contains  $1\leq i \leq \lfloor k/2
\rfloor$ elements of $A$ and $i-1$ elements not in $A$. In this case we say that the components 
are separated by $i$ elements of $A$.
Thus, we define the following partition,
\begin{align*}
\mathcal{B}_3&=\{A\in \operatorname{SG}(2k+2,k)\,|\, \overline{\mathcal{A}} \text{ contains a component of order $3$}\}\\ 
\mathcal{B}_{2,i}&=\{A\in \operatorname{SG}(2k+2,k)\,|\, \overline{\mathcal{A}} \text{ contains two components of order $2$ separated by $i$ elements of $A$}\}.\\
\end{align*}
Notice that any vertex $A\in\mathcal{B}_3$ is characterized by the position of the first element $v$
of the component of order $3$ in $\overline{\mathcal{A}}$. We denote such a vertex as $A=A_{0,v}$. Thus, $|\mathcal{B}_3|=2k+2$.
Similarly, any vertex $A\in \mathcal{B}_{2,i}$ is characterized by the firsts elements, in the clockwise direction of $C_n$, $v$ and $v+2i+1$ of the connected components of $\overline{\mathcal{A}}$
of order $2$. Notice that if  $k$ is even  and $i=k$,
there are $k+1$ such pairs, and that otherwise there are $2k+2$ such pairs. Hence we have 
\begin{align*}
|\mathcal{B}_3|=&2k+2,\\
|\mathcal{B}_{2,i}|=&\begin{cases}
k+1&\text{if $k$ is even and $i=k/2$,}\\
2k+2&\text{otherwise.}
\end{cases}
\end{align*}

We continue by studying the neighborhood of each vertex.
Let $A_0^v=\{v-1,v+3,v+5,\ldots,v+2k-1\}\in \mathcal{B}_3$.
If $B$ is a vertex adjacent to $A_0^v$, then either $B$ contains exactly one element
in $\{v,v+1,v+2\}$ and every element in $\{v+4,v+6,\ldots, v+2k-2,v+2k\}$, or
$B$ contains both $v$ and $v+2$, and $k-2$ elements in $\{v+4,v+6,\ldots, v+2k-2,v+2k\}$.
Thus, the degree of $A_0^v$ is $3+\binom{k-1}{k-2}=k+2$.
If $B$ is contained in $\{v,v+2,v+4,v+6,\ldots, v+2k-2,v+2k\}$, then $B$ is in $\mathcal{B}_3$. More precisely,
if $v+2i$ is not in $B$ for $1\leq i \leq k+1$, then $B=A_0^{v+2i-1}$.  
On the other hand, if $C=\{v+1, v+4,v+6,\ldots, v+2k-2,v+2k\}$ is the other vertex adjacent to $A_0^v$, then $C\in \mathcal{B}_{2,1}$
because the connected components of order $2$ in $\overline{\mathcal{C}}$ are $\{v+2k+1,v\}$ and $\{v+2,v+3\}$.
As this is the only neighbor of $A_0^v$ in $\mathcal{B}_{2,1}$, we denote $A_1^v=\{v,v+2,v+4,v+6,\ldots, v+2k-2,v+2k\}$.
 Thus
\[
N(A_0^v)=\{A_0^{v+2i-1}\,|\,1\leq i \leq k+1\}\cup\{A_1^v\}.
\]

Consider now $A=\{v-1,v+2,v+4,\ldots, v+2i,v+2i+3,v+2i+5,\ldots,v+2k-1\}\in \mathcal{B}_{2,i}$, with $1\leq i \leq \lfloor k/2 \rfloor$. The fact that $\overline{\mathcal{A}}$ has $k$ connected components of
order at most $2$ implies that any vertex $B$ adjacent to $A$ must intersect $\overline{\mathcal{A}}$ in every connected component. Thus that are $4$ such vertices, depending on which element of the
connected components of order $2$ they contain. This means that $|N(A)|=4$.
Furthermore, these neighbors are 
\begin{align*}
B_1&=\{v-2,v+1,v+3,\ldots, v+2i-1,v+2i+2,v+2i+4,\ldots,v+2k-2\}\in\mathcal{B}_{2,i}\\
B_2&=\{v,v+3,v+5,\ldots, v+2i+1,v+2i+4,v+2i+6,\ldots,v+2k\}\in\mathcal{B}_{2,i}\\
B_3&=\{v,v+3,v+5,\ldots, v+2i-1,v+2i+2,v+2i+4,v+2i+6,\ldots,v+2k-2\}\\
B_4&=\{v-2,v+1,v+3,\ldots, v+2i+1,v+2i+4,v+2i+6,\ldots,v+2k\}\\
\end{align*}
Notice that $B_3\in \mathcal{B}_3$ if $i=1$, and $B_3\in \mathcal{B}_{2,i}$ otherwise.
Further, notice that $B_4\in \mathcal{B}_{2,i+1}$ if $i\neq \lfloor k/2\rfloor$. 
This means that given $1\leq i < \lfloor k/2 \rfloor$, every vertex in $\mathcal{B}_{2,i}$ has 
exactly one neighbor in $\mathcal{B}_{2,i+1}$. Thus,
for $1<i\leq \lfloor k/2 \rfloor$, we inductively denote vertex $A_i^v$ 
as the only neighbor of $A_{i-1}^v$ in $\mathcal{B}_{2,i}$.
Let $A_i^u=A=\{v-1,v+2,v+4,\ldots, v+2i,v+2i+3,v+2i+5,\ldots,v+2k-1\}$.
Notice that $B_1=A_i^{u-1}$ and $B_2=A_i^{u+1}$. Further, notice that
when $k$ is even and $i=(k/2)-1$, we have
\begin{align*}
A_{k/2}^u=&B_4\\
=&\{v-2,v+1,v+3,\ldots, v+2i+1,v+2i+4,v+2i+6,\ldots,v+2k-2\}\\
=&\{v-2,v+1,v+3,\ldots, v+k-2+1,v+k-2+4,v+k-2+6,\ldots,v+2k-2\}\\
=&\{v-2,v+1,v+3,\ldots, v+k-1,v+k+2,v+k+4,\ldots,v+2k-2\}\\
=&\{v+k-1,v+k+2,v+k+4,\ldots,v+2k-2,v-2,v+1,v+3,\ldots, v+k-3\}\\
=&\{v-2+k+1,v+1+k+1,v+3+k+1,\ldots,v+k-3+k+1,v+k-1+k+1,\\
&v+k+2+k+1,v+k+4+k+1,\ldots, v+2k-2+k+1\}\\
=&A_{k/2}^{u+k+1}.
\end{align*}
This means that $A_{k/2}^u$ plays the role of $B_4$ for both $A_{(k/2)-1}^u$ and $A_{(k/2)-1}^{u+k+1}$.
It also means that if $k$ is even, $i=k/2$ and $A=A_{k/2}^u$, then $B_4=A_{(k/2)-1}^{u+k+1}$.

We have only left to discuss $B_4$ when $k$ is odd and $i=(k-1)/2$.
Letting $A_{(k-1)/2}^u=A=\{v-1,v+2,v+4,\ldots, v+k-1,v+k+2,v+k+4,\ldots,v+2k-1\}$, 
we have
\begin{align*}
B_4=&\{v-2,v+1,v+3,\ldots, v+2i+1,v+2i+4,v+2i+6,\ldots,v+2k-2\}\\
=&\{v-2,v+1,v+3,\ldots, v+k,v+k+3,v+k+5,\ldots,v+2k-2\}\\
=&\{v+k,v+k+3,v+k+5,\ldots,v+2k-2,v-2,v+1,v+3,\ldots , v+k-2\}\\
=&\{v-1+k+1,v+2+k+1,v+4+k+1,\ldots,v+k-3+k+1,v+k-1+k+1,
\\
&v+k+2+k+1,v+k+4+k+1,\ldots , v+2k-1+k+1\}\\
=&A_{(k-1)/2}^{u+k+1}.
\end{align*}

We can now give a full description of $\operatorname{SG}(2k+2,k)$.
\begin{theorem}
The Schrijver graph $\operatorname{SG}(2k+2,k)$ has vertex set
$V=V_1 \cup V_2$, where 
\begin{align*}
V_1=&\{A_i^v\,|\, 0\leq i \leq \lfloor k/2 \rfloor  -1, 0 \leq v \leq 2k+1\},\\
V_2=&\{A_{\lfloor k/2\rfloor}^v\,|\,  0 \leq v \leq 2k+1 \text{ if $k$ is odd and } 0\leq v \leq k \text{ if $k$ is even}\},\\
\end{align*}
and edge set $E=E_1\cup E_2 \cup E_3$, where
\begin{align*}
E_1=&\big\{ \{A_0^v,A_0^{v+2\ell +1}\}\,|\,  0 \leq v \leq 2k+1, 0\leq \ell \leq k\big\},\\
E_2=&\big\{ \{A_i^v,A_i^{v+1}\}\,|\, 1\leq i \leq \lfloor k/2 \rfloor, 0 \leq v \leq 2k+1 \big\},\\
E_3=&\big\{ \{A_i^v,A_{i+1}^v\}\,|\, 1\leq i \leq \lfloor k/2 \rfloor -1, 0 \leq v \leq 2k+1\big\},\\
E_4=&\big\{ \{A_{\lfloor k/2 \rfloor}^v,A_{\lfloor (k-1)/2\rfloor}^{v+k+1}\}\,|\, 0 \leq v \leq 2k+1 \text{ if $k$ is odd and } 0\leq v \leq k \text{ if $k$ is even} \big\}.\\
\end{align*}
\end{theorem}
\begin{proof}
The result follows from the preceding discussion.
\end{proof}
In order words, $\operatorname{SG}(2k+2,k)$ is the graph obtained as follows. Take the cartesian product $C_{2k+2}\square P_{\lfloor k/2 \rfloor +1}$. First, add edges between vertices that correspond to one end of $P_{\lfloor k/2 \rfloor +1}$ and are
at odd distance in $C_{2k+2}$. Next, take pairs of vertices that correspond to the other end of $P_{\lfloor k/2 \rfloor +1}$ and are at distance $k+1$. If $k$ is odd, add an edge between every pair of such vertices. If $k$ is even, identify each pair of those vertices. The resulting graph is $\operatorname{SG}(2k+2,k)$.

We can now give find the diameter of $\operatorname{SG}(2k+2,k)$. 
We begin by giving the diameter of some induced subgraphs.
\begin{proposition}\label{prop:diamsubgrafos}
\begin{enumerate}
\item The diameter of the subgraph of $\operatorname{SG}(2k+2,k)$
induced by $\mathcal{B}_3$ is $2$.
\item If $k$ is odd, the diameter of the subgraph of $\operatorname{SG}(2k+2,k)$
induced by $\mathcal{B}_{2,(k-1)/2}$
is $(k+1)/2$.
\item If $k$ is even, the diameter of the subgraph of $\operatorname{SG}(2k+2,k)$
induced by $\mathcal{B}_{2,k/2}$ is $(k/2)$.
\end{enumerate}
\end{proposition}

Consider now two vertices $A_i^v,A_j^u$, with $i\leq j$. Using item $1.$ of Proposition \ref{prop:diamsubgrafos} yields
\begin{align*}
\text{dist}(A_i^v,A_j^u)\leq & \text{dist}(A_i^v,A_0^v)+\text{dist}(A_0^v,A_0^u)+\text{dist}(A_0^u,A_j^u)\\
\leq  & i + 2 +j.
\end{align*}
If $k$ is odd, item $2.$ yields 
\begin{align*}
\text{dist}(A_i^v,A_j^u)\leq & \text{dist}(A_i^v,A_{(k-1)/2}^v)+\text{dist}(A_{(k-1)/2}^v,A_{(k-1)/2}^u)+\text{dist}(A_{(k-1)/2}^u,A_j^u)\\
\leq & \frac{k-1}{2}-i+\frac{k+1}{2}+\frac{k-1}{2}-j\\
=&\frac{3k-1}{2}-i-j\\
=&\left\lfloor \frac{3k}{2}\right\rfloor-i-j.
\end{align*}
If $k$ is even, item $3.$ yields
\begin{align*}
\text{dist}(A_i^v,A_j^u)\leq & \text{dist}(A_i^v,A_{k/2}^v)+\text{dist}(A_{k/2}^v,A_{k/2}^u)+\text{dist}(A_{k/2}^u,A_j^u)\\
\leq & \frac{k}{2}-i+\frac{k}{2}+\frac{k}{2}-j\\
=&\frac{3k}{2}-i-j\\
=&\left\lfloor\frac{3k}{2}\right\rfloor-i-j.
\end{align*}
Thus, we have
\[
\operatorname{D}(\operatorname{SG}(2k+2,k))\leq \max_{0\leq i,j\leq \lfloor k/2\rfloor} \min \left\lbrace i+j+2,\left\lfloor\frac{3k}{2}\right\rfloor-i-j\right\rbrace.
\]
As $i+j+2$ and $\left\lfloor\frac{3k}{2}\right\rfloor-i-j$ are respectively an increasing linear function
and a decreasing linear function of $(i+j)$, which meet at
\begin{align*}
i+j+2=&\left\lfloor\frac{3k}{2}\right\rfloor-i-j\\
2(i+j)=&\left\lfloor\frac{3k}{2}\right\rfloor -2\\
i+j=&\frac{1}{2}\left\lfloor\frac{3k}{2}\right\rfloor -1.
\end{align*}
If $\lfloor 3k/2 \rfloor$ is even, then $(1/2)\lfloor 3k/2\rfloor=\lfloor 3k/4\rfloor$, and
\begin{align*}
\operatorname{D}(\operatorname{SG}(2k+2,k))\leq& \max_{0\leq i,j\leq \lfloor k/2\rfloor} \min \left\lbrace i+j+2,\left\lfloor\frac{3k}{2}\right\rfloor-i-j\right\rbrace \\
=&\left\lfloor\frac{3k}{4}\right\rfloor -1+2\\
=&\left\lfloor\frac{3k}{4}\right\rfloor +1.
\end{align*}

If $\lfloor 3k/2 \rfloor$ is odd, then $(1/2)\lfloor 3k/2 \rfloor=\lfloor 3k/4\rfloor +(1/2)$ is not an integer. 
This means that, when $\lfloor 3k/2 \rfloor$ is odd, $\max_{0\leq i,j\leq \lfloor k/2\rfloor} \min \left\lbrace i+j+2,\left\lfloor\frac{3k}{2}\right\rfloor-i-j\right\rbrace$ is either $\lfloor 3k/2\rfloor-i-j$ with $i+j=\lceil 3k/4\rceil -1$ or $i+j+2$ with $i+j=\lfloor 3k/4 \rfloor -1$. But 
\begin{align*}
\left\lfloor \frac{3k}{2}\right\rfloor - \left(\left\lceil\frac{3k}{4}\right\rceil -1\right)=&
\left\lfloor \frac{3k}{2}\right\rfloor - \left\lceil\frac{3k}{4}\right\rceil +1\\
\leq&\left\lfloor\frac{3k}{4}\right\rfloor +1,
\end{align*}
and
\begin{align*}
\left\lfloor\frac{3k}{4}\right\rfloor -1 +2=\left\lfloor\frac{3k}{4}\right\rfloor +1.
\end{align*}

Thus, in any case,
\begin{align*}
\operatorname{D}(\operatorname{SG}(2k+2,k))\leq& \max_{0\leq i,j\leq \lfloor k/2\rfloor} \min \left\lbrace i+j+2,\left\lfloor\frac{3k}{2}\right\rfloor-i-j\right\rbrace \\
=&\left\lfloor\frac{3k}{4}\right\rfloor +1.
\end{align*}

When $k\geq 2$ is even we need to consider two more upper bounds. Consider vertices $A_i^v$ and $A_j^u$ and notice that $A_0^v$ is either 
adjacent to $A_0^u$ or to $A_0^{u+k+1}$. Thus, let 
\[
\ell=\begin{cases}
0 & \text{ if $u-v$ is odd}\\
k+1 & \text{ if $u-v$ is even}
\end{cases}
\]
and notice that $A_0^v$ is adjacent to $A_0^{u+ \ell}$ and that $A_{k/2}^u=A_{k/2}^{u+\ell}$. Thus,
\begin{align*}
\text{dist}(A_i^v,A_j^u)\leq & \text{dist}(A_i^v,A_0^v)+\text{dist}(A_{0}^v,A_{0}^{u+\ell})+\text{dist}(A_{0}^{u+\ell},A_{k/2}^u)+\text{dist}(A_{k/2}^u,A_j^u)\\
\leq & i+1+\frac{k}{2}+\frac{k}{2}-j\\
=& k+i+1-j.
\end{align*}
Notice that, if $j-i=k-\lfloor 3k/4\rfloor -1$, then 
\begin{align*}
k+i+1-j\leq& k+1-\left(k-\left\lfloor \frac{3k}{4}\right\rfloor -1\right)\\
=\left\lfloor \frac{3k}{4}\right\rfloor.
\end{align*}

Finally, without loss of generality,, $0\leq j-i\leq k-\lfloor 3k/4\rfloor -2$ and  $u-v\leq k+1$. Then, we have
\begin{align*}
\text{dist}(A_i^v,A_j^u)\leq & \text{dist}(A_i^v,A_j^v)+\text{dist}(A_j^v,A_j^u)\\
\leq & j-i+u-v.
\end{align*}
Notice that, if $v$ and $u$ are chosen so that $\text{dist}(A_{k/2}^u,A_{k/2}^v)=k/2$, then
$u-v\leq (k/2)+1$. Thus, we get
\begin{align*}
j-i+u-v&\leq k-\left\lfloor \frac{3k}{4}\right\rfloor -2 + \frac{k}{2}+1\\
&=\left\lceil \frac{k}{4}\right\rceil +\frac{k}{2}-1\\
&\leq \left\lfloor \frac{k}{4}\right\rfloor +1 +\frac{k}{2} -1\\
&=\left\lfloor \frac{3k}{4}\right\rfloor.
\end{align*}
Thus, if the previous paths proposed have length at least $\left\lfloor \frac{3k}{4}\right\rfloor +1$,
this last path has length $\left\lfloor \frac{3k}{4}\right\rfloor$.
Therefore, we have
\[
\operatorname{D}(\operatorname{SG}(2k+2,k))\leq \begin{cases}
\left\lfloor \frac{3k}{4}\right\rfloor & \text{ if $k\geq 2$ is even}\\
\left\lfloor \frac{3k}{4}\right\rfloor +1&  \text{otherwise.}
\end{cases}
\]

In order to give two vertices at such distance, let 
\[
v=\begin{cases} 
\frac{k}{2} & \text{ if $k\equiv 0 \pmod 4$}\\
\frac{k+1}{2}+1 & \text{ if $k\equiv 1 \pmod 4$}\\
\frac{k}{2} +1 & \text{ if $k\equiv 2 \pmod 4$}\\
\frac{k+1}{2} & \text{ if $k\equiv 3 \pmod 4$}\\
\end{cases}.
\]
Notice that $\text{dist}(A_0^0,A_0^v)=\operatorname{D}(\mathcal{B}_{3})=2$, as $v$ is even, and that $\text{dist}(A_{\lfloor k/2 \rfloor}^0,A_{\lfloor k/2 \rfloor}^v)=\operatorname{D}(\mathcal{B}_{2,\lfloor k/2\rfloor })=\lfloor (k+1)/2 \rfloor$.
Take vertices $B_1=A_{\lfloor k/2 \rfloor}^0$ and $B_2=A_{\lfloor 3k/4 \rfloor -1- \lfloor k/2 \rfloor}^v$. 
Notice that if $\lfloor 3k/4 \rfloor -1- \lfloor k/2 \rfloor\geq 0$, then $k\geq 3$. 
It is easy to check that any path that goes from $B_1$ to $B_2$
and does not use vertices from $\mathcal{B}_3$ has length at least
\begin{align*}
\lfloor k/2 \rfloor - \left( \lfloor 3k/4 \rfloor - 1 -\lfloor k/2 \rfloor \right) + \lfloor (k+1)/2 \rfloor =& 
2 \lfloor k/2 \rfloor +1- \lfloor 3k/4 \rfloor + \lfloor (k+1)/2 \rfloor \\
=&\lfloor 3k/2 \rfloor - \lfloor 3k/4 \rfloor +1\\
\geq & \left\lfloor\frac{3k}{4}\right\rfloor +1.
\end{align*}
If $k$ is odd it is also easy to check that the shortest path using vertices from $\mathcal{B}_3$ has length
\begin{align*}
\lfloor k/2 \rfloor + \lfloor 3k/4 \rfloor -1- \lfloor k/2 \rfloor +2=&\left\lfloor\frac{3k}{4}\right\rfloor +1.
\end{align*}
On the other hand, if $k\geq 3$ is even, then a shortest path using vertices from $\mathcal{B}_3$ is given by
\[
A_{k/2}^0,A_{(k/2)-1}^{k+1},\ldots, A_1 ^{k+1},A_0^{k+1},A_0^v, A_1^v,\ldots, A_{\lfloor 3k/4 \rfloor -1- (k/2)}^v,
\]
which has length
\[
\frac{k}{2}+1+\left\lfloor \frac{3k}{4} \right\rfloor -1- \frac{k}{2}=\left\lfloor \frac{3k}{4} \right\rfloor.
\]
Hence, we have the following.
\begin{theorem}
\label{th:2k+2}
Let $k\geq 3$. The diameter of $\operatorname{SG}(2k+2,k)$ is $\left\lfloor\frac{3k}{4}\right\rfloor$ if $k$ is even
and $\left\lfloor \frac{3k}{4}\right\rfloor +1$ if $k$ is odd.
\end{theorem}

\section{Conclusions}
In this article we found the exact value of $D\left(\operatorname{SG}(2k+r,k)\right)$ when $r\leq 2$ and when $r\geq k-3$.
We give both an upper and a lower bound for the remaining cases, but evidence suggest they can be improved.
Finding a lower bound for $3\leq r \leq k-4$ that is a function of $k$ would be quite interesting.

Based on Theorem \ref{th:2k+2}
one would think that for other small values of $r$ the bound from Theorem \ref{th:boundm+3} is not tight.
Furthermore, the evidence suggest that $D\left(\operatorname{SG}(2k+r,k)\right)$ is a non-increasing function of $r$.
\begin{conjecture}\label{conj:mono}
If $r\geq 1$, then $D\left(\operatorname{SG}(2k+r,k)\right)\geq D\left(\operatorname{SG}(2k+r+1,k)\right)$.
\end{conjecture}
If Conjecture \ref{conj:mono} is true, then $D\left(\operatorname{SG}(2k+r,k)\right)\leq \lfloor 3k/4\rfloor -(k\hspace*{-0.2cm}\mod 2)$ if $r\geq 2$, which is an improvement on the bound in Theorem \ref{teo-ppal} when $r\leq \lceil 3k/4 \rceil$.
This would also imply that there are more intervals where the diameter remains constant, as it happens when $k-2\leq r \leq 2k-3$ and when $r\geq 2k-2$.
Finding said intervals would be interesting.

Our results allow us to compute (exactly) the diameter of $\operatorname{SG}(n,k)$ for every $2\leq k\le6$ and $n\geq2k+1$. We summarize it in Table \ref{table:diam}. 
\begin{table}[h]
    \centering
    \begin{tabular}{|c|c|c|c|c|c|c|}
\hline
    $D\left(\operatorname{SG}(n,k)\right)=$ & $2$ & $3$ & $4$ & $5$ & $6$ & $7$\\ \hline \hline
 \multirow{2}{3em}{$k=2$} & $n\ge5$ & - & - & - & - & -\\
    & \color{blue}{Obs. \ref{rem:diamd2}, Th. \ref{cor:s2diam2}} &  &  &  &  &\\ \hline

 \multirow{2}{3em}{$k=3$} & $n\ge10$ & $7\le n\le9$ & -& -& - &-\\
    & \color{blue}{Th. \ref{cor:s2diam2}} & \color{blue}{Obs. \ref{rem:diamd2}, Th. \ref{th:diam3}} &  & &  & \\ \hline

\multirow{2}{3em}{$k=4$} & $n\ge14$ & $10\le n\le13$ & $n=9$& -&- &- \\
    & \color{blue}{Th. \ref{cor:s2diam2}} & \color{blue}{Th. \ref{th:diam3}} & \color{blue}{Obs. \ref{rem:diamd2}} & & & \\ \hline

\multirow{2}{3em}{$k=5$} &  $n\ge18$ &  $13\le n\le17$ & $n=12$ & $n=11$ & - & -\\
    &  \color{blue}{Th. \ref{cor:s2diam2}} & \color{blue}{Th. \ref{th:diam3}} & \color{blue}{Cor. \ref{cor:3k-3}} & \color{blue}{Obs. \ref{rem:diamd2}} & & \\ \hline

\multirow{2}{3em}{$k=6$} & $n\ge22$ & $16\le n\le21$ & $n=14,15$ & - & $n=13$ & - \\
    & \color{blue}{Th. \ref{cor:s2diam2}} & \color{blue}{Th. \ref{th:diam3}} & \color{blue}{Cor. \ref{cor:3k-3}, Th. \ref{th:2k+2}} & & \color{blue}{Obs. \ref{rem:diamd2}} & \\ \hline

\multirow{2}{3em}{$k=7$} & $n\ge26$ & $19\le n\le25$ & $n=18$ & $n=17$ & $n=16$ & $n=15$ \\
    & \color{blue}{Th. \ref{cor:s2diam2}} & \color{blue}{Th. \ref{th:diam3}} & \color{blue}{Cor. \ref{cor:3k-3}} & \color{blue}{(*)} & \color{blue}{Th. \ref{th:2k+2}} & \color{blue}{Obs. \ref{rem:diamd2}}\\ \hline
\end{tabular}   
    \caption{Diameter of $\operatorname{SG}(n,k)$ for $2\le k\le7$.\\  {\color{blue}{ (*) }}Computationally computed using SageMath.}
    \label{table:diam}
\end{table}
Notice that Theorem \ref{th:boundm+3} would suggest that $D\left(\operatorname{SG}(n,k)\right)-D\left(\operatorname{SG}(n+1,k)\right)\in \{0,1\}$, which is the case for $k\leq 5$. But Theorem \ref{th:2k+2} lets us compute $D\left(\operatorname{SG}(2k+1,k)\right)-D\left(\operatorname{SG}(2k+2,k)\right)$. When $k\geq 6$ we get
\begin{align*}
D\left(\operatorname{SG}(2k+1,k)\right)-D\left(\operatorname{SG}(2k+2,k)\right)=&k-\left(\left\lfloor\frac{3k}{4}\right\rfloor + (k\hspace*{-0.3cm}\mod 2)\right)\\
=&\left \lceil \frac{k}{4}\right\rceil -(k\hspace*{-0.3cm}\mod 2),\\
\geq& 2.
\end{align*}
This raises the question about whether another such gap could appear as $r$ moves through the interval $[2..k-3]$.
Based on our computations, we propose the following conjecture.
\begin{conjecture}
If $2\leq r \leq k-2$, then $D\left(\operatorname{SG}(2k+r,k)\right)-D\left(\operatorname{SG}(2k+r+1,k)\right)\in \{0,1\}$.
\end{conjecture}




\begin{thebibliography}{10}

\bibitem{Braun10}
B.~Braun.
\newblock Symmetries of the stable Kneser graphs.
\newblock {\em Adv. in Appl. Math.}, 45:12--14, 2010.

\bibitem{Braun11}
B.~Braun.
\newblock Independence complexes of stable Kneser graphs.
\newblock {The Electron. J. of Combin.}, 18, 2011, Paper 118.



\bibitem{GodRoy01}
C.~D.~Godsil, G.~Royle.
\newblock {\em Algebraic graph theory}.
\newblock Graduate Texts in Mathematics. Springer, 2001.

\bibitem{KaSte20}
T.~Kaiser, M.~Stehlik.
\newblock Edge-critical subgraphs of Schrijver graphs.
\newblock {\em J. of Combin. Theory Ser. B}, 144:191--196, 2020.

\bibitem{KaSte22}
T.~Kaiser, M.~Stehlik.
\newblock Edge-critical subgraphs of Schrijver graphs II: The general case.
\newblock {\em J. of Combin. Theory Ser. B}, 152:453--482, 2022.

\bibitem{Lov78}
L.~Lov\'asz.
\newblock Kneser's conjecture, chromatic number and homotopy,
\newblock {\em J. of Combin. Theory, Ser. A}, 25:319-324, 1978.

\bibitem{Meun11}
F.~Meunier.
\newblock The chromatic number of almost stable Kneser hypergraphs.
\newblock {\em J. of Combin. Theory, Ser. A}, 118:1820--1828, 2011.

\bibitem{Sch78}
A. Schrijver.
\newblock Vertex-critical subgraphs of Kneser graphs.
\newblock {\em Nieuw Arch. Wiskd.}, 26(3):454--461, 1978.

\bibitem{SiTa20}
G.~Simonyi, G.~Tardos.
\newblock On $4$-chromatic Schrijver graphs: their structure, non $3$-colorability, and critical edges.
\newblock {\em Acta Math. Hungar.}, 161:583--617, 2020.

\bibitem{Torres15}
P. Torres.
\newblock The automorphism group of the s-stable Kneser graphs.
\newblock {\em Adv. in Appl. Math.}, 89:67--75, 2017. 

\bibitem{T-V17}
P. Torres, M. Valencia-Pabon.
\newblock Shifts of the Stable Kneser Graphs and Hom-Idempotence.
\newblock {\em European J. of Combin.}, 62:50--57, 2017.

\bibitem{V-V05}
M. Valencia-Pabon, J. C. Vera.
\newblock On the diameter of Kneser graphs.
\newblock {\em Discrete Math.}, 305:383--385, 2005.

\end{thebibliography}
\end{document}